\documentclass[10pt,twoside,english,reqno,a4paper]{amsart}

\usepackage{listings,graphicx,amsmath,varioref,amscd,amssymb,color,bm,stmaryrd,amsthm,amsfonts,graphics,geometry,latexsym,pgf,pst-all} 
\theoremstyle{plain}
\usepackage{esint}
\usepackage{amsthm}
\usepackage{tikz}
\usepackage{pgfplots}
\usepackage{enumerate}

\usepackage{resmes} 

\theoremstyle{plain}
\newtheorem{theorem}{Theorem}[section]
\newtheorem{proposition}[theorem]{Proposition}
\newtheorem{lemma}[theorem]{Lemma}

\newtheorem{corollary}[theorem]{Corollary}

\usepackage{geometry}
\usepackage[colorlinks=false]{hyperref}

\theoremstyle{definition}
\newtheorem{defin}[theorem]{Definition}

\newtheorem{remark}[theorem]{Remark}

\theoremstyle{remark}

\usepackage{fouriernc}
\usepackage[T1]{fontenc}

\numberwithin{equation}{section}

\DeclareMathOperator{\R}{\mathbb{R}}
\DeclareMathOperator{\N}{\mathbb{N}}

\newcommand{\car}[1]{\raise1pt\hbox{$\chi$}_{#1}}

\newcommand{\DM }{\mathcal{DM}^\infty }

\def\into{\int_{\Omega}}
\def\intdo{\int_{\partial\Omega}}
\def\dH{\mathrm{d}\mathcal{H}^{N-1}}
\def\ae{\mathrm{a.e.}}

\def\bvo{{BV(\Omega)}}
\def\lio{{L^\infty(\Omega)}}
\def\luo{{L^1(\Omega)}}
\def\lno{{L^N(\Omega)}}

\begin{document}
\title[Existence and comparison results for a doubly singular 1-Laplacian problem with $L^1$ data]{Existence and comparison results for a doubly\\ singular
1-Laplacian problem with $L^1$ data}

\author[A. J. Mart\'inez Aparicio]{Antonio J. Mart\'inez Aparicio}

\address[Antonio J. Mart\'inez Aparicio]{Departamento de Matem\'aticas,
Universidad de Almer\'ia
	\hfill \break\indent
    Ctra. Sacramento s/n, La Ca\^{n}ada de San Urbano, 04120 Almer\'ia, Spain}
\email{\tt ajmaparicio@ual.es}


\begin{abstract}
In this work, we conduct a comprehensive study of problem
\begin{equation*}
	\begin{cases}
		-\Delta_1 u + g(u)|Du| = h(u)f & \text{in }\Omega,\\
		u=0 & \text{on } \partial\Omega,
	\end{cases}
\end{equation*}
where $\Omega\subset \mathbb{R}^N$ is a bounded Lipschitz domain, $f\in L^1(\Omega)$ is a nonnegative datum, and $g,h$ are nonnegative continuous functions on $(0,\infty)$ that may be singular at the origin. Under the minimal assumptions that $g$ is integrable near zero and $h$ is bounded at infinity, we explore the existence of a global $BV(\Omega)$ solution. Furthermore, a comparison principle is proved under suitable monotonicity assumptions on $h$. This framework avoids any growth restrictions on $h$ near the origin, thus allowing for highly singular terms. To handle these nonlinearities, we introduce a novel approach that takes advantage of the rigid structure of the 1-Laplacian operator.
\end{abstract}

\maketitle

\setcounter{tocdepth}{1}
\tableofcontents

\section{Introduction}

The aim of this work is to provide a complete analysis of problem
\begin{equation}
\label{eq:PbIntro}
	\begin{cases}
		-\Delta_1 u + g(u)|Du| = h(u)f & \text{in }\Omega,\\
		u=0 & \text{on } \partial\Omega,
	\end{cases}
\end{equation}
where $\Delta_1 u := \operatorname{div}(|Du|^{-1} Du)$ denotes the 1-Laplacian operator and $\Omega\subset \R^N$ $(N\geq 2)$ is a bounded domain with Lipschitz boundary. We operate under the following fundamental assumptions: $f$ is a nonnegative function in $\luo$ and $g,h\colon [0,\infty) \to [0,\infty]$ are continuous functions which are finite in $(0,\infty)$ and possibly singular at zero. Furthermore, $g$ is strictly positive and integrable at the origin, while $h$ remains bounded at infinity.

The 1-Laplacian operator introduces several technical challenges compared to the $p$-Laplacian ($p>1$). Principally, the lack of reflexivity and compactness of $W^{1,1}(\Omega)$ makes $\bvo$ the natural space for the solutions. Recall that $\bvo$ is the space of $L^1(\Omega)$ functions $u$ whose distributional gradient $Du$ is a finite Radon measure. Since $Du$ is merely a measure, a significant issue arises in defining the quotient $|Du|^{-1} Du$ that appears in the formal expression of $\Delta_1 u$. A rigorous solution concept, now standard in the field, was pioneered in the seminal paper~\cite{AnBaCaMa} and involves a bounded vector field $z$ that plays the role of $|Du|^{-1} Du$. A striking consequence of this definition is the interpretation of the Dirichlet boundary condition, which is satisfied only in a generalized sense, allowing for the existence of solutions with a non-zero trace on $\partial\Omega$.

In the absence of the gradient term ($g\equiv 0$), problem~\eqref{eq:PbIntro}---which arises in several geometric problems and in image processing models, see~\cite{Leo}---exhibits a degenerate behaviour. Unlike what occurs for the $p$-Laplacian, the existence of solution for problem
\begin{equation}
\label{eq:PbIntro2}
	\begin{cases}
		-\Delta_1 u = h(u)f & \text{in }\Omega,\\
		u=0 & \text{on } \partial\Omega,
	\end{cases}
\end{equation}
with $h\equiv 1$ and $f\in\lno \subset W^{-1,\infty}(\Omega)$ strongly depends on the size of $f$ (\cite{MeSeTr}, see also~\cite{CiTr, Kaw}). A solution to~\eqref{eq:PbIntro2} exists only if $\|f\|_{W^{-1,\infty}(\Omega)} \leq 1$, and the unique solution is $u\equiv 0$ whenever $\|f\|_{W^{-1,\infty}(\Omega)} < 1$. It is only in the critical case $\|f\|_{W^{-1,\infty}(\Omega)} = 1$ that nontrivial solutions may appear. However, in this scenario, the problem exhibits a remarkable lack of uniqueness: any increasing transformation of a solution that fixes the origin is also a solution, a consequence of the 1-Laplacian's invariance under such mappings. Furthermore, solutions typically fail to satisfy the boundary condition pointwise, attaining it only in a very weak sense. For instance, if $\Omega$ is the unit ball and $f\equiv N$ (see that $\|f\|_{W^{-1,\infty}(\Omega)} = 1$), then any positive constant $u\equiv \alpha>0$ is a solution to~\eqref{eq:PbIntro2}.

The introduction of the nonlinearity $h$ in~\eqref{eq:PbIntro} may induce a regularizing effect, at least when $h$ is bounded at infinity and $f$ is nonnegative. As shown in~\cite{MaOlPe1}, existence can be established for any $f \in L^N(\Omega)$ if $\lim_{s\to\infty} h(s) = 0$, or even for $f$ merely in $L^1(\Omega)$, provided that $h(s)$ decays as $s^{-1}$ at infinity. These results hold even when $h$ is singular at the origin; indeed, in this case, solutions are strictly positive in $\Omega$ whenever $f>0$.

Motivated by the level set formulation of the inverse mean curvature flow (see~\cite{HuIl}), the authors in~\cite{MaSe} proposed the study of problem
\begin{equation}
\label{eq:PbIntro3}
	\begin{cases}
		-\Delta_1 u + |Du| = f & \text{in }\Omega,\\
		u=0 & \text{on } \partial\Omega,
	\end{cases}
\end{equation}
for a nonnegative datum $f$, which they initially addressed for $f\in L^q(\Omega)$ with $q>N$. Subsequently, this problem was investigated for $f\in L^{N,\infty}(\Omega)$ in~\cite{LaSe1}  and for $f\in L^1(\Omega)$ in~\cite{LaSe2}. The presence of the absorption gradient term exerts a powerful regularizing effect in several respects. First, a global $\bvo$ solution exists for any $f\in L^1(\Omega)$. Furthermore, this solution has no jump part and satisfies the boundary condition pointwise. Finally, as proved in~\cite{LaSe2} using the specific structure of~\eqref{eq:PbIntro3}, a comparison principle holds for this problem, which in turn guarantees the uniqueness of the solution.

A direct extension of~\eqref{eq:PbIntro3} is given by problem~\eqref{eq:PbIntro}. For a nonsingular function $g$, this problem has been investigated in~\cite{BalOP} for $0\leq f\in\luo$, under the assumption that $h(s)$ is bounded at infinity and dominated by a multiple of $s^{-1}$ near the origin. By using the classical strategy of~\cite{BoGa}, and under the assumption
\begin{equation}
    \label{eq:hyp_Intro1}
    \liminf_{s\to\infty} g(s) >0,
\end{equation}
the authors show the existence of a $\bvo$ solution to~\eqref{eq:PbIntro}, which has no jump part and satisfies the zero boundary condition in the sense of traces.

To the author's knowledge, the case of a singular gradient term has only been addressed in~\cite{Bal} under very restrictive conditions. Specifically, it is assumed therein that $f\in L^N(\Omega)$, and that, near the origin, $g(s)$ and $h(s)$ are controlled by $cs^{-\theta}$ and $cs^{-\gamma}$, respectively, where $c, \theta, \gamma>0$ and $\theta+\gamma\leq 1$. Under these assumptions, and further requiring~\eqref{eq:hyp_Intro1} and that $g$ and $h$ are bounded at infinity, the existence of a $\bvo$ solution to~\eqref{eq:PbIntro} with no jump part and zero boundary trace is established in~\cite{Bal}.

Without being exhaustive, other recent contributions to related 1-Laplacian problems can be found in~\cite{AbDaSe, FiPi, MaOlPe2, MoSe, PiSaSt, SaSe}. Regarding $p$-Laplacian problems with a singular natural gradient term, we refer the reader to the pioneering works~\cite{ABLP, ACLMOP}.

In the present article, problem~\eqref{eq:PbIntro} is thoroughly investigated, extending the existing literature in several directions. The premise of this work is to study in depth problem~\eqref{eq:PbIntro} under a broad and challenging framework. Concretely, it is assumed that:
\begin{enumerate}[i)]
    \item $f$ is a nonnegative function in $L^1(\Omega)$.
    \item $g,h\colon [0,\infty)\to [0,\infty]$ are continuous functions, finite outside the origin and possibly singular at zero, such that
    \begin{itemize}
        \item[--] $g$ is positive and remains integrable near zero, that is, $\lim_{s\to 0^+} \int_s^1 g(t)\, \mathrm{d}t < \infty$,
        \item[--] $h$ is bounded at infinity, that is, $\limsup_{s\to\infty} h(s) < \infty$.
    \end{itemize}
\end{enumerate}

It must be highlighted that, in this setting, $g$ may be unbounded at infinity. Furthermore, no growth restriction is imposed on $h$ near zero, while the only requirement for $g$ near the origin is the aforementioned integrability condition. Finally, it is worth noting that problem~\eqref{eq:PbIntro3} is naturally recovered as a particular case within this framework.

Regarding the existence of a global $BV(\Omega)$ solution, we consider two general cases: either $g$ degenerates at infinity or it does not. In the first case, we show that~\eqref{eq:hyp_Intro1} is sufficient to guarantee the existence of a solution to~\eqref{eq:PbIntro}, which represents a major improvement over the results in~\cite{Bal, BalOP}. In the second case, where $g$ does not satisfy~\eqref{eq:hyp_Intro1}, we prove that a solution to~\eqref{eq:PbIntro} exists provided that $h$ decays at least as $s^{-1}$ at infinity. To the author's knowledge, this is the first time that such a regularizing effect induced by $h$ has been observed for problems like~\eqref{eq:PbIntro}. This finding is somewhat surprising since, for $h\equiv 1$, a strong nonexistence phenomenon may arise (see the discussion in Section~\ref{sec:Weak_sol}).

In any case, the growth of $h$ near the origin is not controlled, and $g$ is only required to be integrable near zero. To deal with such arbitrary singular nonlinearities, we introduce a novel approach that takes advantage of the rigid structure of the 1-Laplacian operator, which can be extrapolated to other types of problems.

With respect to the uniqueness of solution, we show that a comparison principle for problem~\eqref{eq:PbIntro} holds whenever $h$ is a nonincreasing function. This generalizes the comparison principle established in~\cite{LaSe2} for problem~\eqref{eq:PbIntro3}.

Our study is further complemented by the exploration of additional properties of~\eqref{eq:PbIntro}. We provide several sufficient conditions to ensure the boundedness of the solutions; in particular, we show that this property is guaranteed if $f$ is integrable enough or if $h$ vanishes at some point. Furthermore, we investigate a weaker notion of solution, which accounts for cases where $u$ may not belong to the space $BV(\Omega)$.

The paper is organized as follows. In Section~\ref{sec:Prelim}, we introduce the technical tools required throughout the work, including a brief overview of $BV$ functions and Anzellotti's theory. In Section~\ref{sec:Main}, we present our main results concerning the existence of $BV(\Omega)$ solutions and the comparison principle for problem~\eqref{eq:PbIntro}. The existence results are proved in Section~\ref{sec:Pf_exist}, while the comparison principle is established in Section~\ref{sec:Pf_Comp}. Section~\ref{sec:Bound} is devoted to providing sufficient conditions that ensure the boundedness of the solutions. Finally, in Section~\ref{sec:Weak_sol}, we explore the existence under a weaker concept of solution, and in Section~\ref{sec:Extensions} we discuss several potential extensions of our results.

\section{Preliminaries}
\label{sec:Prelim}

\subsection{Notation}

Throughout this article, $\Omega$ is an open bounded subset of $\R^N$ ($N\geq 2$) with Lipschitz boundary. Given a set $E\subset \R^N$, $\mathcal{H}^{N-1}(\partial E)$ stands for the $(N-1)$-dimensional Hausdorff measure of its boundary, while $|E|$ denotes the classical $N$-dimensional Lebesgue measure $\mathcal{L}^N(E)$.

We say that $f_1<f_2$ in $E$ if $f_1(x)<f_2(x)$ for almost every $x\in E$. As for the integrals, we denote by $\into f$ the integral of $f$ with respect to the Lebesgue measure, whereas $\into f\mu$ denotes the integral of $f$ with respect to the measure $\mu$.

For a fixed $k>0$, we set $T_{k}$ and $G_{k}$ as the real functions defined for any $s\in \R$ as
\[
T_k(s):=\max \{-k,\min \{s,k\} \} \qquad \text{and} \qquad G_k(s):= \min\{s+k, \max\{0,s-k\}\}. 
\]
Note in particular that $T_k(s) + G_k(s)=s$ for every $s\in \mathbb{R}$. Moreover, we represent the characteristic function of a set $A\subseteq\Omega$ by $\chi_A$.

Finally, let us clarify that we use the term \textit{increasing} in the strict sense ($g(s_1)<g(s_2)$ for $s_1<s_2$), whereas \textit{nondecreasing} refers to the case where the inequality is not necessarily strict. An analogous convention applies to \textit{decreasing} and \textit{nonincreasing} functions.

\subsection{Functions of bounded variation} We denote by $\mathcal{M}(\Omega)$ the space of Radon measures with finite total variation over $\Omega$. The space of functions of bounded  variation is defined as
\[
BV(\Omega):= \left\{ u\in L^1(\Omega) : Du \in \mathcal{M}(\Omega)^N \right\},
\]
where $Du$ stands for the distributional gradient of $u$. Moreover, we denote by $BV_{\rm loc} (\Omega)$ the set of functions $u\in L^1_{\rm loc}(\Omega)$ that belong to $BV(\omega)$ for any open set $\omega$ compactly contained in $\Omega$. In the following, we summarize the fundamental properties of this space, mainly extracted from~\cite{AFP}.

First, every function in $\bvo$ has a trace in $L^1(\partial\Omega)$. This allows to define the norm
\[
\|u\|_{BV(\Omega)}=\int_\Omega|Du| + \int_{\partial\Omega} |u| \ \dH,\ \forall u\in\bvo,
\]
where $\int_\Omega|Du|$ denotes the total variation of the measure $Du$ over $\Omega$. We highlight that both the norm $\|\cdot\|_\bvo$ and the functional $u\mapsto \int_\Omega \varphi|Du|$ (for a fixed $0\leq \varphi\in C_c^1(\Omega)$) are lower semicontinuous with respect to the $L^1(\Omega)$ convergence.

Recall that for $u \in L_{\rm loc}^1(\Omega)$, we say that $u$ has an approximate limit at $x_0 \in \Omega$ if there exists $z\in \R$ such that 
\begin{equation*}
    \lim_{r\to 0^+} \frac{1}{|B_r(x_0)|}\int_{B_r(x_0)} |u(y) - z| \, \mathrm{d}y = 0.
\end{equation*}
The set $S_u$ of points where this property does not hold is called the approximate discontinuity set. This set $S_u$ is a $\mathcal{L}^N$-negligible Borel set (\cite[Proposition~3.64]{AFP}). For any $x_0\in \Omega\setminus S_u$ the number $z$, which is uniquely determined, is called the approximate limit of $u$ at $x_0$ and is denoted by $\widetilde{u}(x_0)=z$.

On the other hand, we say that $x_0\in\Omega$ is an approximate jump point of $u$ if there exist $a,b\in \R$ and a unit vector $\nu \in S^{N-1}$ such that $a\neq b$ and
\begin{align*}
    \lim_{r\to 0^+} \frac{1}{|B^+_r(x_0)|} \int_{B^+_{r}(x_0)} |u(y) - a| \, \mathrm{d}y &= 0, \\
    \lim_{r\to 0^+} \frac{1}{|B^-_r(x_0)|} \int_{B^-_{r}(x_0)} |u(y) - b| \, \mathrm{d}y &= 0,
\end{align*}
where $B_r^\pm(x_0) := \{y\in B_r(x_0): \pm (y-x_0)\cdot \nu >0\}$. The triplet $(a,b,\nu)$, uniquely determined up to a permutation of $(a,b)$ and a change of sign of $\nu$, is denoted by $(u^+(x_0), u^-(x_0), \nu_u(x_0))$. The set of approximate jump points is denoted by $J_u$. It is a Borel subset of $S_u$ (\cite[Proposition~3.69]{AFP}), and $\mathcal{H}^{N-1}(S_u \setminus J_u) = 0$ if $u \in BV(\Omega)$. Moreover, up to an $\mathcal{H}^{N-1}$-negligible set, $J_u$ is an $\mathcal{H}^{N-1}$-rectifiable set, and an orientation $\nu_u(x)$ is defined for $\mathcal{H}^{N-1}$-almost every $x \in J_u$ (\cite[Theorem~3.78]{AFP}).

The precise representative of $u$  is defined as the $\mathcal{H}^{N-1}$-a.e. finite function 
$u^*\colon \Omega\setminus(S_u\setminus J_u)\to \mathbb{R}$ given by 
\begin{equation*}
u^*(x)=
\begin{cases}
\widetilde{u}(x)& \text{ if } x\in \Omega\setminus S_u,\\
\frac{u^+(x)+u^-(x)}{2}& \text{ if } x\in J_u.
\end{cases}
\end{equation*}
The terminology for $u^*$ comes from the fact that, if $\rho_\varepsilon$ is a sequence of standard mollifiers, then $u \ast \rho_\varepsilon$ pointwise converges to $u^*$ in its domain. In the particular case where $\mathcal{H}^{N-1}(J_u)=0$, we write $u$ instead of $u^*$ for the precise representative of $u$, as no ambiguity is possible when integrating against a measure which is absolutely continuous with respect to $\mathcal{H}^{N-1}$. Hereafter, for simplicity, we write $u(x)=\widetilde u(x)$ whenever $x\in\Omega\setminus S_u$.

If $u\in\bvo$, then $Du$ can be decomposed as
\[
Du = D^a u + D^s u,
\]
where $D^a u$ is the absolutely continuous part of $Du$ with respect to the Lebesgue measure $\mathcal{L}^N$, and $D^s u$ is the singular part. The measure $D^a u$ is just $\nabla u \, \mathcal{L}^N$, where $\nabla u$ is the approximate gradient of $u$ (\cite[Definition~3.70]{AFP}). The singular part can be further decomposed as the sum of its Cantor and jump part as
\[
D^s u = D^c u + D^j u, \qquad D^c u := D^s u \resmes (\Omega\setminus S_u), \qquad D^j u:= D^s u \resmes J_u = (u^+ - u^-)\nu_u \mathcal{H}^{N-1} \resmes J_u.
\]
The sum $D^a u + D^c u$ is called the diffuse part of $Du$, and is denoted by $\widetilde D u$. An important feature of $\widetilde D u$ is that, for every $t\in\R$, it holds that $\widetilde D u \resmes \{u=t\} = 0$.

A chain rule formula for these functions is available.

\begin{lemma}{\cite[Theorem 3.96]{AFP}}
\label{lem:ChainRule_AFP}
Let $u\in BV_{\rm loc}(\Omega)$. If $\Psi \colon \R\to \R$ is a Lipschitz function, then $\Psi (u)\in BV_{\rm loc}(\Omega)$ and
\begin{equation}
\label{eq:chain_rule_1}
D\Psi(u) = \Psi'(u) \widetilde D u + (\Psi(u^+) - \Psi(u^-)) \nu_u \mathcal{H}^{N-1} \resmes J_u.
\end{equation}
\end{lemma}

\begin{remark}
The chain rule can be rewritten in a more concise form. With this aim, given a locally integrable function $\beta\colon \R\to\R$ and $v\in BV_{\rm loc}(\Omega)$, we set
\begin{equation}
\label{eq:def_beta_hash}
\beta(v)^\# := \begin{cases}
    \frac{1}{v^+ - v^-} \int_{v^-}^{v^+} \beta(s) \, \mathrm{d}s, & \text{ if } x\in J_v,\\
    \beta(v), & \text{ otherwise}.
\end{cases}
\end{equation}
Observe that $\beta(v)^\#$ is a particular representative of $\beta(v)$ and coincides with $\beta(v)^*$ if $\beta$ is an affine function, but in general $\beta(v)^\# \neq \beta(v)^*$. By means of this function, the chain rule~\eqref{eq:chain_rule_1} can be expressed as
\begin{equation}
\label{eq:chain_rule_2}
D\Psi(u) = \Psi'(u)^\# D u.
\end{equation}
\end{remark}

In this work, we mainly deal with functions without a jump part. In such cases, a more general version of the chain rule holds, provided that the composition is known to belong to $BV_{\rm loc}(\Omega)$. In essence, this has been shown in~\cite[Lemma~2.4]{GOP}, although we add some clarifications.

\begin{lemma}
\label{lem:ChainRule}
Let $0\leq u\in BV_{\rm loc}(\Omega)$ be such that $D^j u = 0$, and let $g\colon [0,\infty) \to (0,\infty]$ be a continuous function finite outside the origin. Let $\Gamma\colon [0,\infty) \to [-\infty, \infty)$ be such that $\Gamma'(s) = g(s)$, and suppose that $\Gamma(u)\in BV_{\rm loc}(\Omega)$. Then $g(u) \in L^1_{\rm loc}(\Omega, |Du|)$ and it holds
\[
\chi_{\{u>0\}}|D\Gamma(u)| = g(u) |Du| \text{ as measures in } \Omega.
\]
Consequently, if $\Gamma(0)$ is finite, then
\[
|D\Gamma(u)| = g(u) |Du| \text{ as measures in } \Omega.
\]
\end{lemma}

\begin{proof}
In~\cite[Lemma~2.4]{GOP}, the authors use a truncation argument and the classical chain rule to show that
\[
\chi_{\{u>0\}}|D\Gamma(u)| = g(u) \chi_{\{u>0\}} |Du| \text{ as measures in } \Omega.
\]
To prove this, it is essential to use that $\Gamma(u)\in BV_{\rm loc}(\Omega)$. Since $D^j u =0$, then $Du$ only has a diffuse part. Then, as a consequence of~\cite[Proposition 3.92]{AFP}, $D u \resmes \{u=t\} = 0$ for every $t\in\R$; in particular, $Du$ vanishes at $\{u=0\}$. Therefore, one has
\[
g(u) \chi_{\{u>0\}} |Du| = g(u) |Du|.
\]
If $\Gamma(0)$ is finite, taking into account that $\Gamma(u)\in BV_{\rm loc}(\Omega)$ and has no jump part, one can apply again~\cite[Proposition 3.92]{AFP} to ensure that $D\Gamma(u)$ vanishes at $\{\Gamma(u)=\Gamma(0)\}$. Since $\Gamma$ is increasing, then $\{u=0\}$ is equal to $\{\Gamma(u)=\Gamma(0)\}$. This ends the proof.
\end{proof}

Finally, we recall that the classical embedding $W_0^{1,1}(\Omega) \hookrightarrow L^{\frac{N}{N-1}}(\Omega)$ extends to the space $\bvo$ by means of a density argument. Then, it holds that
\begin{equation}
\label{eq:sob_embed}
\|u\|_{L^\frac{N}{N-1}(\Omega)} \leq \mathcal{S}_1 \|u\|_\bvo,\ \forall u\in\bvo,
\end{equation}
where $\mathcal{S}_1$ denotes the best constant of the embedding. Note that $1^*=\frac{N}{N-1}$ is the critical Sobolev exponent.

\subsection{$L^\infty$-divergence-measure vector fields}

These vector fields play a fundamental role in problems involving the 1-Laplacian operator. With the goal of establishing a generalized Gauss-Green formula, this theory began to be independently developed by Anzellotti (\cite{Anz}) and Chen and Frid (\cite{ChenFrid}). To be precise, let us define the set
\[
\DM(\Omega):= \big\{ z\in L^\infty(\Omega)^N : \operatorname{div}z \in \mathcal{M}(\Omega) \big\},
\]
and let us denote by $\DM_{\rm loc}(\Omega)$ its local version, i.e., the set of bounded vector fields $z$ with $z\in \mathcal{M}_{\rm loc} (\Omega)$. It is worth noting that, as proved in~\cite[Proposition~3.1]{ChenFrid}, $\operatorname{div}z$ is absolutely continuous with respect to the measure $\mathcal{H}^{N-1}$ whenever $z\in \DM (\Omega)$. Given a vector field $z\in \DM_{\rm loc} (\Omega)$ and a function $u\in BV_{\rm loc}(\Omega) \cap L^\infty_{\rm loc}(\Omega)$, we define the distribution $(z,Du) \colon C_c^\infty(\Omega) \to \R$ as
\begin{equation} \label{eq:Pairing}
\langle(z,Du),\varphi\rangle:=-\int_\Omega u^*\varphi\operatorname{div}z-\int_\Omega uz\cdot\nabla\varphi, \,\quad \forall \varphi\in C_c^\infty(\Omega),
\end{equation}
This distribution defines a generalized dot product between $z$ and $Du$ and is, in fact, a Radon measure that has, locally, finite total variation (\cite{Ca}). Furthermore, for every Borel set $B$ and every open set $U$ with $B \subset U \subset \Omega$ it holds
\begin{equation}
\label{eq:AbsCont}
    \left| \int_B (z, Du) \right| \leq \int_B |(z, Du)| \leq \|z\|_{L^{\infty}(U)^N} \int_B |Du|.
\end{equation}
The measure $(z,Du)$ actually generalizes the usual dot product. It admits a decomposition into two parts with respect to the $N$-dimensional Lebesgue measure $\mathcal{L}^N$: an absolutely continuous part and a singular one. The absolutely continuous part, denoted by $(z,Du)^a$, satisfies that $(z,Du)^a = z\cdot D^a u = z\cdot \nabla u\, \mathcal{L}^N$ (see~\cite[Theorem~4.12]{CrDec}). Then, for functions $\varphi \in W^{1,1}(\Omega)$, the pairing $(z,D\varphi)$ just becomes $(z,D\varphi) = z\cdot \nabla \varphi\, \mathcal{L}^N$.

As a direct consequence of~\eqref{eq:AbsCont}, $(z,Du)$ is absolutely continuous with respect to the measure $|Du|$. Denoting by \(\theta(z, Du, \cdot)\)
the Radon-Nikodym derivative of $(z, Du)$ with respect to $|Du|$, it holds
\begin{equation*}
	(z, Du) = \theta(z,Du,x) \, {|Du|} \text{ as measures in } \Omega.
\end{equation*}
It can be deduced that this derivative behaves well with respect to the composition, as the following result shows.

\begin{lemma}
    \label{lem:Composition}
	Let $z \in \DM_{\rm loc}(\Omega)$ and $u \in BV_{\rm loc}(\Omega) \cap \lio$. Let $\Psi: I \to \mathbb{R}$ be a continuous nondecreasing function, with $I \subseteq \mathbb{R}$ an interval. If $u(\Omega)\subseteq I$ and 
	$\Psi(u) \in BV_{\rm loc}(\Omega) \cap \lio$, then 
	\[
    \theta(z, D\Psi(u),x) = \theta(z, Du,x) \quad \text{for } |D\Psi(u)|\text{-}\ae\  x \in \Omega.
	\]
    As a consequence, $(z,Du)=|Du|$ as measures implies $(z,D\Psi(u))= |D\Psi(u)|$ as measures.
\end{lemma}

\begin{remark}
\label{rem:Composition_incr}
If $\Psi$ is increasing, then $\theta(z, D\Psi(u),x) = \theta(z, Du,x)$ for $|Du|$-${\rm a.e.}\ x \in \Omega$. In this case, $(z,Du)=|Du|$ if and only if $(z,D\Psi(u))= |D\Psi(u)|$.
\end{remark}

\begin{remark}
As a consequence of the chain rule (Lemma~
\ref{lem:ChainRule_AFP}), a sufficient condition for $\Psi(u)$ to belong to $BV_{\rm loc}(\Omega) \cap \lio$ is that $\Psi$ is a Lipschitz function. In this particular case, Lemma~\ref{lem:Composition} has been shown in~\cite[Proposition $4.5$]{CrDec}.
\end{remark}

\begin{proof}[Proof of Lemma~\ref{lem:Composition}]
The proof relies on a classical argument from \cite{Anz}. First, we address the case where $\Psi$ is increasing. We denote the level sets of $u$ and $\Psi(u)$ by
\[
E_{u,t} := \{x \in \Omega : u(x) > t\} \qquad \text{ and } \qquad E_{\Psi(u), \Psi(t)} := \{x \in \Omega : \Psi(u(x)) > \Psi(t)\}.
\]
Since $\Psi$ is increasing, these sets are identical for every $t \in \mathbb{R}$. Consequently, it holds that
\begin{equation} \label{eq:Pf_L_Comp_1}
    D\chi_{E_{u,t}} = D\chi_{E_{\Psi(u), \Psi(t)}}.
\end{equation}
Taking into account that $u, \Psi(u) \in BV_{\rm loc}(\Omega) \cap L^\infty(\Omega)$, an application of \cite[Theorem~4.2]{CrDec} leads to
\begin{gather}
    \label{eq:Pf_L_Comp_2}
    \theta(z, Du, x) = \theta(z, D\chi_{E_{u,t}}, x) \quad \text{for } |D\chi_{E_{u,t}}|\text{-a.e. } x \in \Omega, \\
    \label{eq:Pf_L_Comp_3}
    \theta(z, D\Psi(u), x) = \theta(z, D\chi_{E_{\Psi(u), \Psi(t)}}, x) \quad \text{for } |D\chi_{E_{\Psi(u), \Psi(t)}}| \text{-a.e. } x \in \Omega,
\end{gather}
for almost every $t \in \mathbb{R}$.
Therefore, gathering together \eqref{eq:Pf_L_Comp_1}, \eqref{eq:Pf_L_Comp_2} and \eqref{eq:Pf_L_Comp_3}, for almost every $t\in\R$ we have that
\begin{equation}
\label{eq:Pf_L_Comp_4}
\theta(z, Du, x) = \theta(z, D\chi_{E_{u,t}}, x) = \theta(z, D\chi_{E_{\Psi(u), \Psi(t)}},x) = \theta(z, D\Psi(u), x),
\end{equation}
for $|D\chi_{E_{u,t}}|$-$\ae$ $x \in \Omega$. Thanks to coarea formula for the total variation,~\eqref{eq:Pf_L_Comp_4} also holds for $|Du|$-a.e. $x \in \Omega$.

Finally, the case of a nondecreasing function $\Psi$ can be handled as in~\cite[Proposition~2.7]{LaSe2}, using an approximation argument on the results already proved for increasing compositions.
\end{proof}

Another result regarding compositions is given below. It has been shown in~\cite[Proposition~3.3]{LaSe2} when $D^j u =0$. For completeness, we include here its proof in the case of a general $BV$-function.

\begin{lemma}
	\label{lem:Composition2}
	Let $z\in \DM_{\rm loc}(\Omega)$ and $u\in BV_{\rm loc}(\Omega)$ be such that $(z,DT_k(u)) = |DT_k(u)|$ for every $k>0$. Let $\Psi: I \to \mathbb{R}$ be a continuous nondecreasing function, with $I \subseteq \mathbb{R}$ an interval. If $u(\Omega)\subseteq I$ and 
	$\Psi(u) \in BV_{\rm loc}(\Omega) \cap L_{\rm loc}^\infty(\Omega)$, then 
	\[
	(z,D\Psi(u)) = |D\Psi(u)| \text{ as measures.}
	\]
\end{lemma}

\begin{proof}
Let $\varphi\in C_c^1(\Omega)$ be nonnegative. Using Lemma~\ref{lem:Composition}, we obtain that $(z,D\Psi(T_k(u))) = |D\Psi(T_k(u))|$ for every $k>0$. On the one hand, as $\Psi(u)\in BV_{\rm loc}(\Omega)$, the Monotone Convergence Theorem gives that
\[
\lim_{k\to \infty} \into \varphi |D\Psi(T_k(u))| = \lim_{k\to \infty} \int_{\{u< k\}} \varphi |D\Psi(u)| = \into \varphi |D\Psi(u)|.
\]
On the other hand, taking into account that $\Psi(u)\in BV_{\rm loc}(\Omega) \cap L^\infty_{\rm loc}(\Omega)$ and~\eqref{eq:Pairing}, we can use the Lebesgue Theorem to obtain that
\begin{align*}
\lim_{k\to \infty} \into \varphi (z,D\Psi(T_k(u)) &= \lim_{k\to \infty} \left(-\int_\Omega \Psi(T_k(u))^*\varphi\operatorname{div}z-\int_\Omega \Psi(T_k(u)) z\cdot\nabla\varphi \right)\\
&= -\int_\Omega \Psi(u)^*\varphi\operatorname{div}z-\int_\Omega \Psi(u) z\cdot\nabla\varphi = \into \varphi (z,D\Psi(u)).
\end{align*}
This concludes the proof.
\end{proof}

A general Leibniz rule for the pairing also holds. In~\cite[Proposition~4.11]{CrDec}, it is shown that if $z\in \DM_{\rm loc}(\Omega)$ and $u,v\in BV_{\rm loc}(\Omega)\cap L^\infty_{\rm loc} (\Omega)$, then
\begin{equation}
\label{eq:Prod}
    (z, D(uv)) = u^*(z,Dv) + v^*(z,Du).
\end{equation}

Since $\Omega$ has a Lipschitz boundary, the outward normal unit vector $\nu(x)$ is defined for $\mathcal H^{N-1}$-almost every $x\in\partial\Omega$. In~\cite{Anz}, it is shown that every $z \in \mathcal{DM}^{\infty}(\Omega)$ possesses
a weak trace on $\partial \Omega$ of the
normal component of $z$ which is denoted by
$[z, \nu]$. This notion of weak trace generalizes the classical one, in the sense that $[z, \nu] = z\cdot \nu$ on $\partial\Omega$ if $z\in C^1(\overline{\Omega}, \R^N)$. Moreover, it verifies
\begin{equation*}
\| [z,\nu] \|_{L^\infty (\partial\Omega)}\leq \|z\|_{L^\infty(\Omega)^N}.
\end{equation*}

A Green-type formula involving the pairing $(z,Du)$ and the weak trace $[z,\nu]$ is stated as follows. 

\begin{proposition}{\cite[Proposition~2.6]{DecGiSe}}
Let $z\in \DM_{\rm loc}(\Omega)$ and let $u\in BV(\Omega) \cap \lio$ be such that $u^*\in L^1(\Omega, \operatorname{div}z)$. Then $uz\in \DM(\Omega)$ and it holds
\begin{equation} \label{eq:Green}
	\int_{\Omega} u^* \operatorname{div}z + \int_{\Omega} (z, Du) =
	\int_{\partial \Omega} [uz, \nu] \ \dH.
\end{equation}
\end{proposition}

In the case where $z\in\DM(\Omega)$ and $u\in\bvo\cap\lio$, the normal trace satisfies (see~\cite[Lemma~5.6]{Ca})
\[
[uz,\nu] = u[z,\nu] \ \ \mathcal{H}^{N-1}\text{-}\ae \text{ on } \partial\Omega.
\]
If the regularity is relaxed, i.e, if $z\in\DM_{\rm loc}(\Omega)$ and $u\in\bvo\cap\lio$ are such that $uz\in\DM(\Omega)$, then it holds (see~\cite[Proposition~2.7]{DecGiSe})
\[
\big| [uz,\nu] \big| \leq |u| \cdot \|z\|_{L^\infty(\Omega)^N} \ \ \mathcal{H}^{N-1}\text{-}\ae \text{ on } \partial\Omega.
\]

\subsection{A technical lemma on the jump part of the solutions}

To prove that solutions to~\eqref{eq:PbMain} have no jump part, we follow a strategy introduced in~\cite{BalOP} that involves the definition of a pairing different (though similar) from~\eqref{eq:Pairing}.

Let $\beta\colon \R\to\R$ be a locally Lipschitz function. Given $z\in \DM_{\rm loc}(\Omega)$ and $v\in BV_{\rm loc}(\Omega)$ with $\beta(v)\in BV_{\rm loc}(\Omega) \cap L_{\rm loc}^\infty(\Omega)$, we define the distribution $\left(z,D\beta(v)^\# \right) \colon C_c^\infty(\Omega) \to \R$ as
\begin{equation}
\label{eq:new_Pairing}
\left\langle \left(z,D\beta(v)^\# \right),\varphi \right\rangle:=-\int_\Omega \beta(v)^\# \varphi\operatorname{div}z-\int_\Omega \beta(v) z\cdot\nabla\varphi, \,\quad \forall \varphi\in C_c^\infty(\Omega),
\end{equation}
where $\beta(v)^\#$ is the representative of $\beta(v)$ defined in~\eqref{eq:def_beta_hash}.

In~\cite[Lemma~2.5]{MaSe} it is proved that, when $v\in BV_{\rm loc}(\Omega) \cap L^\infty_{\rm loc}(\Omega)$ and $\beta$ is a $C^1$ function, the pairing $\left(z,D\beta(v)^\# \right)$ is a Radon measure satisfying, for each Borel set $B$ and each open set $U$ such that $B\subset U\subset \Omega$, that
\begin{equation*}
    \left| \int_B \left(z, D\beta(v)^\# \right) \right| \leq \int_B \left| \left(z, D\beta(v)^\# \right) \right| \leq \|z\|_{L^{\infty}(U)^N} \int_B |D\beta(v)|.
\end{equation*}
The method employed in~\cite{MaSe} can be adapted, with only minor modifications, to show that this property also holds if $v\in BV_{\rm loc}(\Omega)$ and $\beta$ is merely a locally Lipschitz function, provided that $\beta(v)\in BV_{\rm loc}(\Omega) \cap L_{\rm loc}^\infty(\Omega)$.

The next result states that functions $v\in BV_{\rm loc}(\Omega)$ verifying a distributional inequality involving its gradient have no jump part.

\begin{lemma}{\cite[Lemma~2.3]{BalOP}}
\label{lem:no_jump}
Let $\beta \colon \R\to \R$ be a locally Lipschitz, increasing function. Let $z\in \DM_{\rm loc} (\Omega)$ be such that ${\|z\|_{\lio^N}\leq 1}$, and let $v$ satisfy $v\in BV_{\rm loc}(\Omega)$ and $\beta(v)\in BV_{\rm loc}(\Omega) \cap L^\infty_{\rm loc}(\Omega)$. Given $f\in L^1_{\rm loc}(\Omega)$, assume that
\begin{equation}
\label{eq:No_jump_H}
-\operatorname{div}z + |Dv| \leq f \qquad \text{and} \qquad \left(z, D\beta(v)^\# \right) = |D\beta(v)|.
\end{equation}
Then $D^j v =0$.
\end{lemma}

\begin{remark}
\label{rem:no_jump}
As noted in the proof of~\cite[Theorem~5.3]{BalOP}, this result remains valid under a slight variation of~\eqref{eq:No_jump_H}. More precisely, assuming $v\geq 0$, Lemma~\ref{lem:no_jump} also holds if~\eqref{eq:No_jump_H} is replaced by the conditions $-\chi_{\{v>0\}}^* \operatorname{div}z + |Dv| \leq f$ and $\chi_{\{v>0\}}^* \left(z, D\beta(v)^\# \right) = \chi_{\{v>0\}}^* |D\beta(v)|$.
\end{remark}

\section{Main results}
\label{sec:Main}

In this section, we present several novel existence and uniqueness results concerning the problem
\begin{equation}
	\label{eq:PbMain}
	\begin{cases}
		-\Delta_1 u + g(u)|Du| = h(u)f & \text{in }\Omega,\\
		u=0 & \text{on } \partial\Omega,
	\end{cases}
\end{equation}
where $f>0$ belongs to $\luo$. We always assume that $g\colon [0,\infty) \to (0,\infty]$ is a continuous function, finite outside the origin, which is integrable near zero, i.e.,
\begin{equation}
    \label{eq:hyp_g_int}
    \lim_{s\to 0^+} \int_s^1 g(t)\, \mathrm{d}t < \infty.
\end{equation}
Notice that~\eqref{eq:hyp_g_int} allows $g$ to have mild singularities near the origin. For instance, $g(s)$ may behave near zero as $s^{-\theta}$ with $\theta\in[0,1)$. In what follows, we denote
\[
\Gamma(s) = \int_0^s g(t) \, \mathrm{d}t,\ \forall s\geq 0.
\]
Observe that $\Gamma$ is increasing and, since we always assume~\eqref{eq:hyp_g_int}, it holds that $\Gamma(0)=0$.

On the other hand, we always suppose that $h\colon [0,\infty) \to [0,\infty]$ is a continuous function, finite outside the origin, such that
\begin{equation}
    \label{eq:hyp_h_inf}
    h(\infty) := \limsup_{s\to\infty} h(s) < \infty.
\end{equation}
We stress that no control of $h$ near zero is imposed; consequently, $h$ may exhibit any growth at the origin. Typical examples verifying this condition include $h(s) = s^{-\gamma}$ with $\gamma\geq 0$ or $h(s) = e^\frac{1}{s}$.

We explicitly remark that these assumptions cover the case where $g$ and $h$ are finite at the origin. In particular, both $g$ and $h$ may be constant functions.

For this problem, the concept of solution is the following.

\begin{defin}
\label{def:sol}
Let $0<f\in\luo$. A nonnegative function $u\in BV(\Omega)$ is a \textit{solution} to problem~\eqref{eq:PbMain} if $D^j u=0$, $\Gamma(u)\in \bvo$, $h(u)f \in \luo$ and if there exists $z\in \DM(\Omega)$ with $\|z\|_{L^\infty(\Omega)^N}\leq 1$ such that
\begin{gather} 
    \label{eq:def_dist}
    -\operatorname{div}z + g(u)|Du| = h(u)f \text{ as measures in } \Omega,\\
    \label{eq:def_pair}
    (z,DT_k(u))=|DT_k(u)| \text{ as measures in } \Omega \text{ for any } k>0,\\
    \label{eq:def_bord}
    u(x)=0 \text{ for  $\mathcal{H}^{N-1}$-a.e. } x \in \partial\Omega.		
\end{gather}
\end{defin}

\begin{remark}
Since $D^ju=0$ and $\Gamma(u)\in\bvo$, Lemma~\ref{lem:ChainRule} implies that $g(u)|Du|$ is a finite Radon measure and that the identity $|D\Gamma(u)| = g(u)|Du|$ holds.
\end{remark}

\begin{remark}
If $h$ is singular at zero (i.e. $h(0)=\infty$), then any solution $u$ to~\eqref{eq:PbMain} is strictly positive in $\Omega$; indeed, since $h(u)f \in \luo$ and $f>0$, the set $\{u=0\}$ must have measure zero. Conversely, if $h(0)<\infty$, problem~\eqref{eq:PbMain} may exhibit a degenerate behaviour, allowing for the existence of trivial solutions even for nontrivial datum $f$ (see~\cite[Proposition~4.4]{LaSe1}).
\end{remark}

In the above definition, the vector field $z$ plays the role of the quotient $\frac{Du}{|Du|}$ that formally appears in the 1-Laplacian operator. This approach was first introduced in the pioneering article~\cite{AnBaCaMa} and, since then, has been widely adopted in the literature. In~\eqref{eq:def_dist}, $z$ \textit{replaces} $\frac{Du}{|Du|}$ within the divergence operator, while~\eqref{eq:def_pair} asserts that $z$ actually \textit{is} the quotient $\frac{Du}{|Du|}$.

Note also that solutions to~\eqref{eq:PbMain} attain the Dirichlet boundary condition pointwise. This regularity is due to the presence of the absorption gradient term. In its absence, the boundary condition must be understood in a very weak sense, and solutions with a non-zero boundary trace may arise (see, for instance,~\cite{DecGOP, MeSeTr}).

\subsection{Existence of finite energy solutions}

We first study the case where $g$ does not degenerate at infinity, i.e., we suppose that $g$ satisfies
\begin{equation}
    \label{eq:hyp_g_inf}
    \liminf_{s\to\infty} g(s) >0.
\end{equation}
We point out that $g$ can be unbounded at infinity. Some model functions verifying both~\eqref{eq:hyp_g_int} and~\eqref{eq:hyp_g_inf} are $g\equiv 1$ or $g(s) = 1+s^{-\theta}$ with $\theta\in(0,1)$.

The use of the structure condition~\eqref{eq:hyp_g_inf} in problems involving natural growth terms to obtain finite energy solutions is widespread in the literature, dating back to the seminal work~\cite{BoGa}. In the context of 1-Laplace problems as~\eqref{eq:PbMain}, this assumption is natural and appears in~\cite{Bal, BalOP, LaSe1, LaSe2}.

Our first existence result is the following.

\begin{theorem}
\label{th:Exist}
Let $0<f\in L^1(\Omega)$. Assume that $g$ verifies~\eqref{eq:hyp_g_int} and that $h$ satisfies~\eqref{eq:hyp_h_inf}. If~\eqref{eq:hyp_g_inf} holds, then problem~\eqref{eq:PbMain} has a nonnegative solution $u\in BV(\Omega)$.
\end{theorem}

Our result generalizes the existing literature in several respects. First, it extends the results of~\cite{BalOP}, where the existence of solution is established when $g$ is nonsingular and $h$ is controlled from above near zero by a function of the form $cs^{-\gamma}$, with $c>0$ and $\gamma\in[0,1]$. In contrast, Theorem~\ref{th:Exist} shows that existence holds independently of the behaviour of $h$ near the origin. 

On the other hand, until now, the existence of a solution to~\eqref{eq:PbMain} when $g$ is singular at zero was only known under the assumption that $f\in L^N(\Omega)$ and that, near the origin, $g(s)\leq cs^{-\theta}$ and $h(s)\leq cs^{-\gamma}$, where $c,\theta,\gamma>0$ and $\theta+\gamma\leq 1$. Within the framework of singular gradient terms, Theorem~\ref{th:Exist} represents a major improvement, since $f$ is allowed to belong to $\luo$ and the only assumption near zero is the integrability condition~\eqref{eq:hyp_g_int}.

The situation becomes more delicate when assumption~\eqref{eq:hyp_g_inf} is dropped and $g$ degenerates at infinity. To the author's knowledge, this case has only been previously investigated when $g$ is nonsingular and $h\equiv 1$ in~\cite{LaSe1}. In this setting, the velocity at which $g$ tends to 0 at infinity plays a decisive role. If $g$ is not integrable at infinity, there exists a unique solution to~\eqref{eq:PbMain} in a weaker sense than that of Definition~\ref{def:sol}; however, this solution generally does not belong to $BV(\Omega)$. By contrast, if $g$ is integrable at infinity, a nonexistence phenomenon arises when $W^{-1,\infty}(\Omega)$-norm of $f$ is sufficiently large (see Section~\ref{sec:Weak_sol} for further discussion).

The next result shows that a global $BV$ solution to~\eqref{eq:PbMain} always exists provided that $h(s)$ decays at least like $s^{-1}$ at infinity, independently of the behaviour of $g(s)$ as $s$ diverges. Furthermore, if $f$ is more summable, the decay assumption on $h$ at infinity can be relaxed.

\begin{theorem}
\label{th:Exist2}
Let $0<f\in L^m(\Omega)$ with $m\geq 1$. Assume that $g$ verifies~\eqref{eq:hyp_g_int} and that $h$ satisfies~\eqref{eq:hyp_h_inf}. Then problem~\eqref{eq:PbMain} has a nonnegative solution $u\in BV(\Omega)$ if one of the following cases holds:
\begin{enumerate}
    \item[(i)] $m=1$ and $\limsup_{s\to \infty} h(s)s < \infty$,
    \item[(ii)] $m = \left(\frac{1^*}{1-\theta}\right)' = \frac{N}{(N-1)\gamma + 1}$ and $\limsup_{s\to \infty} h(s)s^{\gamma} < \infty$ for some $\gamma\in (0,1)$,
    \item[(iii)] $m=N$ and $\|f\|_\lno < (\mathcal{S}_1 h(\infty))^{-1}$, where $\mathcal{S}_1$ is defined in~\eqref{eq:sob_embed}.
\end{enumerate}
\end{theorem}

As far as we know, this is the first work to observe the regularizing effect induced by the nonlinearity $h$ in problems involving gradient terms such as~\eqref{eq:PbMain}. When $g\equiv 0$, a similar result is shown in~\cite{MaOlPe1}. We emphasize once again that no assumption on $g$ at infinity is required.

It should be noted that, when $h$ vanishes at some point, a bounded solution to~\eqref{eq:PbMain} always exists without requiring further assumptions beyond~\eqref{eq:hyp_g_int}. This is shown in Section~\ref{sec:Bound_h=0} by combining Theorem~\ref{th:Exist2} with a truncation argument.

The strategy to prove both existence theorems is to approximate~\eqref{eq:PbMain} by removing the singularities and then passing to the limit. Once the a priori estimates have been obtained, the passage to the limit is identical in both frameworks. Full details are provided in Section~\ref{sec:Pf_exist}.

\subsection{Comparison principle and uniqueness of solution}

We establish a comparison principle for problem~\eqref{eq:PbMain} if, in addition to~\eqref{eq:hyp_g_int}, it also holds that
\begin{equation}
    \label{eq:hyp_comp}
    h(s) \text{ is nonincreasing for } s>0.
\end{equation}
To the author's knowledge, the only comparison principle for~\eqref{eq:PbMain} available in the literature is the one in~\cite{LaSe2}, which is proved for $g\equiv h \equiv 1$. This case is covered by our assumptions and, consequently, our result generalizes that of~\cite{LaSe2}. The comparison principle is stated as follows.

\begin{theorem}
\label{th:Comp}
Let $0<f_1,f_2\in L^1(\Omega)$ be such that $f_1\leq f_2$ $\ae$ in $\Omega$. Assume that $g$ verifies~\eqref{eq:hyp_g_int} and that $h$ satisfies~\eqref{eq:hyp_comp}. If $u_i\in\bvo$ ($i=1,2$) are solutions to problems
\begin{equation}
	\label{eq:PbComp}
	\begin{cases}
		-\Delta_1 u_i + g(u_i)|Du_i| = h(u_i)f_i & \text{in }\Omega,\\
		u_i=0 & \text{on } \partial\Omega,
	\end{cases}
\end{equation}
then $u_1\leq u_2$ $\ae$ in $\Omega$.
\end{theorem}

As an immediate consequence, we obtain the next uniqueness result.

\begin{corollary}
Let $0<f\in L^1(\Omega)$. Assume that $g$ verifies~\eqref{eq:hyp_g_int} and that $h$ satisfies~\eqref{eq:hyp_comp}. Then problem~\eqref{eq:PbMain} has at most one solution $u\in \bvo$.
\end{corollary}

This corollary can be combined with Theorems~\ref{th:Exist} and~\ref{th:Exist2} to ensure the existence and the uniqueness of solution provided that~\eqref{eq:hyp_comp} holds.

\begin{corollary}
If, in addition to the assumptions of Theorem~\ref{th:Exist} or Theorem~\ref{th:Exist2}, condition~\eqref{eq:hyp_comp} also holds, then problem~\eqref{eq:PbMain} has a unique solution $u\in \bvo$.
\end{corollary}

The proof of Theorem~\ref{th:Comp} follows the approach developed in~\cite{LaSe2}. The full details are provided in Section~\ref{sec:Pf_Comp}.

\section{Proof of the existence theorems}
\label{sec:Pf_exist}

In the following, we include the proof of Theorems~\ref{th:Exist} and~\ref{th:Exist2}. First, we introduce the approximate problems and obtain uniform estimates depending on the assumptions imposed on $g$ and $h$. Then, we pass to the limit to obtain a solution to~\eqref{eq:PbMain}. We emphasise that, in this last step, we introduce a novel approach that allows us to deal with the general singularities of $g$ and $h$ in a straightforward way.

\subsection{Approximation scheme and uniform estimates}
\label{sec:Approx}

We define, for every $s\geq 0$ and $n\in\N$, the functions
\[
g_n(s)= T_n \left(g(s)+\tfrac{1}{n}\right) \qquad \text{and} \qquad h_n(s) = T_n(h(s)),
\]
and we set $\Gamma_n(s):= \int_0^s g_n(t)\, \mathrm{d}t$.

We consider the approximated problems
\begin{equation}
	\label{eq:PbApprox}
	\begin{cases}
		-\Delta_1 u_n + g_n(u_n)|Du_n| = h_n(u_n)f & \text{in }\Omega,\\
		u_n=0 & \text{on } \partial\Omega.
	\end{cases}
\end{equation}
Since $g_n$ is nonsingular and $\liminf_{s\to\infty} g_n(s) \geq \frac{1}{n}$, then~\cite[Theorem~4.4]{BalOP} applies to guarantee the existence of a nonnegative solution $u_n\in \bvo$ to~\eqref{eq:PbApprox} (in the sense of Definition~\ref{def:sol}). In what follows, we denote by $z_n\in \DM(\Omega)$ the vector field associated with $u_n$.

Taking $\phi \in \bvo\cap\lio$ as test function in the weak formulation of~\eqref{eq:PbApprox}, one obtains
\[
- \into \phi \operatorname{div}z_n + \into g_n(u_n) \phi |Du_n| = \into h_n(u_n) \phi f.
\]
Since $z_n \in \DM(\Omega)$, the Green identity~\eqref{eq:Green} gives that
\begin{equation}
\label{eq:app_weak_form}
\into (z_n, D\phi) - \intdo \phi [z_n,\nu] \ \dH + \into g_n(u_n) \phi |Du_n| = \into h_n(u_n) \phi f, \ \forall \phi \in \bvo\cap\lio.
\end{equation}

In the following, we prove several uniform estimates for~\eqref{eq:PbApprox} under suitable assumptions on $g$ and $h$. First, we address the right-hand side term of~\eqref{eq:PbApprox}.

\begin{lemma}
\label{lem:bound_hf}
Assume that $g$ verifies~\eqref{eq:hyp_g_int} and that $h$ satisfies~\eqref{eq:hyp_h_inf}. Then there exists $C>0$ such that
\[
\into h_n(u_n)f < C,\ \forall n\in\N.
\]
\end{lemma}

\begin{proof}
We take $e^{-\Gamma_n(u_n)} \in \bvo \cap \lio$ as test function in~\eqref{eq:app_weak_form}, obtaining that
\begin{equation}
    \label{eq:Pf_Lem_AP_1}
    \into \left(z_n, De^{-\Gamma_n(u_n)}\right) - \intdo e^{-\Gamma_n(u_n)} [z_n,\nu] \ \dH  + \into g_n(u_n) e^{-\Gamma_n(u_n)} |Du_n| = \into h_n(u_n) e^{-\Gamma_n(u_n)} f.
\end{equation}
Since $s\mapsto e^{-\Gamma_n(s)}$ is a nonincreasing function and $e^{-\Gamma_n(u_n)} \in \bvo \cap \lio$, Lemma~\ref{lem:Composition2} can be applied to deduce that $\left(z_n, De^{-\Gamma_n(u_n)}\right) = -\left| De^{-\Gamma_n(u_n)} \right|$. As $u_n$ has no jump part, the chain rule implies that $\left| De^{-\Gamma_n(u_n)} \right| = g_n(u_n) e^{-\Gamma_n(u_n)} |Du_n|$. Using also that $u_n\lvert_{\partial\Omega} =0$,~\eqref{eq:Pf_Lem_AP_1} becomes
\begin{equation}
\label{eq:Pf_Lem_AP_2}
    \into h_n(u_n) e^{-\Gamma_n(u_n)} f = - \intdo [z_n,\nu] \ \dH.
\end{equation}
Observe that
\begin{equation}
\label{eq:Pf_Lem_AP_3}
    e^{-\Gamma_n(1)} \int_{\{u_n\leq 1\}} h_n(u_n) f \leq \int_{\{u_n\leq 1\}} h_n(u_n) e^{-\Gamma_n(u_n)} f \leq \into h_n(u_n) e^{-\Gamma_n(u_n)} f.
\end{equation}
Joining~\eqref{eq:Pf_Lem_AP_2} and~\eqref{eq:Pf_Lem_AP_3}, we deduce that
\begin{equation}
\label{eq:Pf_Lem_AP_4}
    \int_{\{u_n\leq 1\}} h_n(u_n) f \leq - e^{\Gamma_n(1)} \intdo [z_n,\nu] \ \dH \leq  e^{\Gamma_n(1)} \mathcal{H}^{N-1}(\partial \Omega) \leq e^{\Gamma(1)+1} \mathcal{H}^{N-1}(\partial \Omega).
\end{equation}
where we have used that $\| [z_n,\nu] \|_{L^\infty(\partial\Omega)} \leq 1$. We stress that $e^{\Gamma(1)}$ is finite because $g$ satisfies~\eqref{eq:hyp_g_int}.

Finally, observe that
\begin{equation}
\label{eq:Pf_Lem_AP_5}
    \int_{\{u_n> 1\}} h_n(u_n) f \leq \sup_{s>1} h(s) \|f\|_\luo,
\end{equation}
and $\sup_{s>1} h(s)$ is finite due to~\eqref{eq:hyp_h_inf}. Our proof concludes just adding~\eqref{eq:Pf_Lem_AP_4} and~\eqref{eq:Pf_Lem_AP_5}.
\end{proof}

Next, we make use of Lemma~\ref{lem:bound_hf} to handle the gradient term in the approximated problems~\eqref{eq:PbApprox}.

\begin{lemma}
\label{lem:bound_grad}
Assume that $g$ verifies~\eqref{eq:hyp_g_int} and that $h$ satisfies~\eqref{eq:hyp_h_inf}. Then there exists $C>0$ such that
\[
\|\Gamma_n(u_n)\|_\bvo < C,\ \forall n\in\N.
\]
\end{lemma}

\begin{proof}
As $\Gamma_n(u_n)\lvert_{\partial\Omega} =0$ for every $n\in\N$, it suffices to prove that the total variation of $\Gamma_n(u_n)$ is uniformly bounded. With this aim, we take 1 as test function in~\eqref{eq:app_weak_form} to obtain that
\[
- \intdo [z_n,\nu] \ \dH + \into g_n(u_n) |Du_n| = \into h_n(u_n)f.
\]
Since $\| [z_n,\nu] \|_{L^\infty(\partial\Omega)} \leq 1$ and the sequence $h_n(u_n)f$ is bounded in $L^1(\Omega)$ by Lemma~\ref{lem:bound_hf}, we deduce that $\into g_n(u_n) |Du_n|$ is uniformly bounded. As $u_n$ has no jump part, then $g_n(u_n) |Du_n| = |D\Gamma_n(u_n)|$ and thus our claim follows.
\end{proof}

So far, the estimates that we have obtained are valid in the frameworks of both Theorems~\ref{th:Exist} and~\ref{th:Exist2}. The next result asserts that, if the hypotheses of Theorem~\ref{th:Exist} are satisfied, then the solutions $u_n$ are uniformly bounded in $BV(\Omega)$.

\begin{lemma}
\label{lem:bound_sol}
Assume that $g$ verifies~\eqref{eq:hyp_g_int} and that $h$ satisfies~\eqref{eq:hyp_h_inf}. If~\eqref{eq:hyp_g_inf} holds, then there exists $C>0$ such that
\[
\|u_n\|_\bvo <C,\ \forall n\in\N.
\]
\end{lemma}

\begin{proof}
For $k>0$, we take $T_k(u_n) \in \bvo \cap \lio$ as test function in~\eqref{eq:app_weak_form} to deduce that
\[
\into (z_n, DT_k(u_n)) - \intdo T_k(u_n) [z_n,\nu] \ \dH + \into T_k(u_n) g_n(u_n) |Du_n| = \into T_k(u_n) h_n(u_n)f.
\]
Since $(z_n, DT_k(u_n)) = |DT_k(u_n)|$ and $u_n$ has zero boundary trace, we obtain
\begin{equation}
\label{eq:Pf_Lem_AP_6}
    \into |DT_k(u_n)| + \into T_k(u_n) g_n(u_n) |Du_n| = \into T_k(u_n) h_n(u_n)f.
\end{equation}
As $g$ satisfies~\eqref{eq:hyp_g_inf}, we can take $k>0$ such that $\inf_{s>k} g(s)>\eta$ for some $\eta\in (0,1)$. Then, one has
\begin{equation}
\label{eq:Pf_Lem_AP_7}
    k\eta \int_{\{u_n>k\}} |Du_n| \leq k \int_{\{u_n>k\}} g_n(u_n) |Du_n| \leq \into T_k(u_n) g_n(u_n) |Du_n|.
\end{equation}
Combining~\eqref{eq:Pf_Lem_AP_6} and~\eqref{eq:Pf_Lem_AP_7}, and using that $\into |DT_k(u_n)| = \int_{\{u_n\leq k\}} |Du_n|$ because $u_n$ has no jump part, it follows
\[
\min\{1,k\eta\} \into |Du_n| \leq 
\int_{\{u_n\leq k\}} |Du_n| + k\eta \int_{\{u_n>k\}} |Du_n| \leq \into T_k(u_n) h_n(u_n)f \leq k \into h_n(u_n)f.\]
By Lemma~\ref{lem:bound_hf}, the sequence $h_n(u_n)f$ is bounded in $\luo$, and then $\into |Du_n|$ is uniformly bounded. In addition, since $u_n\lvert_{\partial\Omega} = 0$ for all $n\in\N$, we conclude that the sequence $u_n$ is bounded in $\bvo$.
\end{proof}

Finally, it remains to prove that $u_n$ is also uniformly bounded in $\bvo$ under the hypotheses of Theorem~\ref{th:Exist2}. To this end, we follow some ideas of~\cite{MaOlPe1}.

\begin{lemma}
\label{lem:bound_sol2}
Assume that $g$ verifies~\eqref{eq:hyp_g_int}, that $h$ satisfies~\eqref{eq:hyp_h_inf}, and that $f$ belongs to $L^m(\Omega)$ $(m\geq 1)$. Then there exists $C>0$ such that
\[
\|u_n\|_\bvo <C,\ \forall n\in\N,
\]
if one of the following cases occurs:
\begin{enumerate}
    \item[(i)] $m=1$ and $\limsup_{s\to \infty} h(s)s < \infty$,
    \item[(ii)] $m = \left(\frac{1^*}{1-\gamma}\right)' = \frac{N}{(N-1)\gamma + 1}$ and $\limsup_{s\to \infty} h(s)s^{\gamma} < \infty$ for some $\gamma\in (0,1)$,
    \item[(iii)] $m=N$ and $\|f\|_\lno < (\mathcal{S}_1  h(\infty))^{-1}$, where $\mathcal{S}_1$ is defined in~\eqref{eq:sob_embed}.
\end{enumerate}
\end{lemma}

\begin{proof}
First, we note that $u_n= T_k(u_n) + G_k(u_n)$ for any $k>0$. Therefore, to show that $u_n$ is bounded in $\bvo$, it suffices to show that the total variations of both $T_k(u_n)$ and $G_k(u_n)$ remain uniformly bounded for some fixed $k>0$ (recall that $u_n\lvert_{\partial\Omega} = 0$).

We begin by addressing $T_k(u_n)$, which can be estimated in the same way for all three cases. Just as in the proof of Lemma~\ref{lem:bound_sol}, we take $T_k(u_n)$ as test function in~\eqref{eq:PbApprox} to obtain, after dropping a nonnegative term in~\eqref{eq:Pf_Lem_AP_6}, that
\[
\into |DT_k(u_n)| \leq \into T_k(u_n) h_n(u_n)f.
\]
Since the sequence $h_n(u_n)f$ is bounded in $\luo$ by Lemma~\ref{lem:bound_hf}, there is some $C>0$ such that
\begin{equation}
\label{eq:Pf_Lem_AP_8}
\into |DT_k(u_n)| \leq Ck,\ \forall n\in\N.
\end{equation}

To handle $G_k(u_n)$, we fix $\ell >k>0$ and we take $G_k(T_\ell (u_n)) \in \bvo\cap \lio$ as test function in~\eqref{eq:app_weak_form}. Taking into account that $u_n$ has zero boundary trace, we obtain that 
\[
\into (z_n, DG_k(T_\ell (u_n))) + \into g_n(u_n) G_k(T_\ell (u_n)) |Du_n| = \into h_n(u_n) G_k(T_\ell (u_n)) f.
\]
Using Lemma~\ref{lem:Composition} and dropping the second integral, it follows that
\begin{equation}
\label{eq:Pf_Lem_AP_9}
\into |DG_k(T_\ell (u_n))| \leq \into h_n(u_n) G_k(T_\ell (u_n)) f.
\end{equation}
From this point, we consider the three cases separately:
\medskip

\textbf{Case (i).} Since $\limsup_{s\to \infty} h(s)s < \infty$, there exists some $C(k)>0$ such that $\sup_{s>k} h(s)s < C(k)$. Hence, as $G_k(s)=0$ when $s\leq k$ and $G_k(s)\leq s$ for any $s\geq 0$, in~\eqref{eq:Pf_Lem_AP_9} we obtain that
\[
\into |DG_k(T_\ell (u_n))| \leq \int_{\{u_n>k\}} h_n(u_n) u_n f \leq C(k) \|f\|_\luo.
\]
Using now the lower semicontinuity of the total variation with respect to the $L^1$-convergence (observe that $G_k(T_\ell (u_n))\to G_k(u_n)$ in $\luo$ as $\ell\to\infty$), we conclude that
\begin{equation*}
\into |DG_k(u_n)| \leq C(k) \|f\|_\luo,\ \forall n\in\N.
\end{equation*}

\textbf{Case (ii).} Since $\limsup_{s\to \infty} h(s)s^{\gamma} < \infty$, there is some $C(k)>0$ such that $\sup_{s>k} h(s)s^{\gamma} < C(k)$. In this way, using that $G_k(s)=0$ when $s\leq k$ and $G_k(s)\leq s$ for any $s\geq 0$, in~\eqref{eq:Pf_Lem_AP_9} we deduce that
\[
\into |DG_k(T_\ell (u_n))| \leq \int_{\{u_n>k\}} h_n(u_n) u_n^\gamma G_k(T_\ell (u_n))^{1-\gamma} f \leq C(k) \int_{\{u_n>k\}} G_k(T_\ell (u_n))^{1-\gamma} f.
\]
Applying H\"older's inequality and the Sobolev inequality~\eqref{eq:sob_embed} to the right-hand side of the previous expression, we obtain that
\[
\into |DG_k(T_\ell (u_n))| \leq C(k) \|f\|_{L^m(\Omega)} \left(\into G_k(T_\ell (u_n))^{m'(1-\gamma)} \right)^\frac{1}{m'} \leq C(k) \mathcal{S}_1 \|f\|_{L^m(\Omega)} \left(\into |DG_k(T_\ell (u_n))| \right)^\frac{1^*}{m'},
\]
where, in the last inequality, we have used that $m'(1-\gamma) = 1^*$ and that $u_n\lvert_{\partial\Omega} = 0$. Since $\frac{1^*}{m'} = 1-\gamma$, it follows that
\[
\into |DG_k(T_\ell (u_n))| \leq \left( C(k) \mathcal{S}_1 \|f\|_{L^m(\Omega)} \right)^\frac{1}{\gamma},
\]
and, using as before the lower semicontinuity of the total variation, we conclude that
\[
\into |DG_k (u_n)| \leq \left( C(k) \mathcal{S}_1 \|f\|_{L^m(\Omega)} \right)^\frac{1}{\gamma},\ \forall n\in\N.
\]

\textbf{Case (iii).} Since~\eqref{eq:hyp_h_inf} holds, then $\sup_{s>k} h(s)$ is finite for any $k>0$. As $G_k(s)=0$ for any $s\leq k$, from~\eqref{eq:Pf_Lem_AP_9} we obtain that
\begin{equation*}
\into |DG_k(T_\ell (u_n))| \leq \sup_{s>k}h(s)\into G_k(T_\ell (u_n)) f.
\end{equation*}
Using H\"older's inequality and the Sobolev inequality~\eqref{eq:sob_embed}, and taking into account that $u_n\lvert_{\partial\Omega} = 0$, we deduce that
\begin{equation*}
\into |DG_k(T_\ell (u_n))| \leq \sup_{s>k}h(s) \|f\|_\lno \|G_k(T_\ell (u_n))\|_{L^{\frac{N}{N-1}}(\Omega)} \leq \sup_{s>k}h(s) \|f\|_\lno \mathcal{S}_1 \into |DG_k(T_\ell (u_n))|.
\end{equation*}
Since we are assuming that $\|f\|_\lno < (\mathcal{S}_1 h(\infty))^{-1}$, we can take $k_0>0$ large such that
\[
\sup_{s>k_0}h(s) \|f\|_\lno \mathcal{S}_1 < 1.
\]
Therefore, we conclude that
\[
\into |DG_{k_0}(T_\ell (u_n))| = 0, \ \forall n\in \N,
\]
that is, $G_{k_0}(T_\ell (u_n)) = 0$ for any $n\in\N$. Since $\ell>k_0$, this implies that $\|u_n\|_{\lio} \leq k_0$ for any $n\in\N$. As we had already shown~\eqref{eq:Pf_Lem_AP_8}, this ends the proof.
\end{proof}

\begin{remark}
Note that in the latter case we have indeed proved that $u_n$ is bounded in $\lio$. Hence, under these assumptions, the solution we find in Theorem~\ref{th:Exist2} is bounded. The boundedness of the solutions is further discussed in Section~\ref{sec:Bound}.
\end{remark}

\subsection{Passage to the limit}

Once these estimates have been established, we are now in a position to prove both Theorems~\ref{th:Exist} and~\ref{th:Exist2}. As previously noted, the passage to the limit is identical in both frameworks. The unified proof is as follows.

\begin{proof}[Proof of Theorems~\ref{th:Exist} and~\ref{th:Exist2}]
By virtue of Lemma~\ref{lem:bound_sol} (in the framework of Theorem~\ref{th:Exist}) or Lemma~\ref{lem:bound_sol2} (in the framework of Theorem~\ref{th:Exist2}), there exists a subsequence of $u_n$ (not relabelled) and some nonnegative $u\in \bvo$ such that
\begin{equation*}
    u_n\to u \text{ in } \luo \qquad \text{ and } \qquad u_n\to u \ \ae \text{ in } \Omega.
\end{equation*}

Moreover, from Lemma~\ref{lem:bound_grad} we deduce that $\Gamma_n(u_n)$ has a subsequence (not relabelled) such that $\Gamma_n(u_n)$ converges in $\luo$ and $\ae$ in $\Omega$ to some function belonging to $\bvo$. As $u_n\to u$ $\ae$ in $\Omega$, the limit of $\Gamma_n(u_n)$ must be $\Gamma(u)$. Therefore, $\Gamma(u)\in \bvo$.

In addition, since $\into h_n(u_n)f$ is bounded by Lemma~\ref{lem:bound_hf}, the Fatou Lemma implies that $h(u)f\in \luo$. Consequently, if $h(0)=\infty$, it follows that $u>0$ $\ae$ in $\Omega$ because $f>0$.

In the following, for the sake of clarity, we divide the proof into several steps.
\medskip 

\textbf{Step 1.} Existence of the vector field $z$.
\smallskip

As $\|z_n\|_{L^\infty(\Omega)^N}\le 1$, then there exists $z\in L^\infty(\Omega)^N$ such that, up to a subsequence, $z_n\rightharpoonup z$ *-weakly in $L^\infty(\Omega)^N$. Since the norm is *-weakly lower semicontinuous, then we also have $\|z\|_{L^\infty(\Omega)^N}\leq 1$.
\medskip

\textbf{Step 2.} An auxiliary distributional inequality.
\smallskip

We want to prove that
\begin{equation}
\label{eq:prov3}
-\operatorname{div} \left( e^{-\Gamma(u)}z \right) \geq h(u) e^{-\Gamma(u)} f \text{ in } \mathcal{D}'(\Omega).
\end{equation}
Note that from~\eqref{eq:prov3}, it follows that $e^{-\Gamma(u)}z\in\DM_{\rm loc}(\Omega)$. To establish inequality~\eqref{eq:prov3}, for $0\leq \varphi \in C_c^1(\Omega)$, we take $e^{-\Gamma_n(u_n)} \varphi \in \bvo\cap \lio$ as test function in~\eqref{eq:app_weak_form}, which yields after using the Leibniz rule~\eqref{eq:Prod} to
\begin{equation}
\label{eq:prov1}
\into \left(z_n, De^{-\Gamma_n(u_n)}\right) \varphi + \into e^{-\Gamma_n(u_n)} z_n \cdot \nabla \varphi + \into g_n(u_n) e^{-\Gamma_n(u_n)} \varphi |Du_n| = \into h_n(u_n) e^{-\Gamma_n(u_n)} f \varphi.
\end{equation}
Since $s\mapsto e^{-\Gamma_n(s)}$ is nonincreasing and $e^{-\Gamma_n(u_n)} \in \bvo \cap \lio$, we can apply Lemma~\ref{lem:Composition2} to ensure that $\left(z_n, De^{-\Gamma_n(u_n)} \right) = -\left| De^{-\Gamma_n(u_n)} \right|$. Moreover, since $D^j u_n = 0$, it follows that $\left| De^{-\Gamma_n(u_n)} \right| = g_n(u_n) e^{-\Gamma_n(u_n)} |Du_n|$. Cancelling terms in~\eqref{eq:prov1}, one obtains
\begin{equation}
\label{eq:prov2}
\into e^{-\Gamma_n(u_n)} z_n \cdot \nabla \varphi = \into h_n(u_n) e^{-\Gamma_n(u_n)} f \varphi.
\end{equation}
Now we pass to the limit in~\eqref{eq:prov2}. With respect to the first integral in~\eqref{eq:prov2}, the passage to the limit is immediate because $z_n\rightharpoonup z$ *-weakly in $L^\infty(\Omega)^N$ and $e^{-\Gamma_n(u_n)} \to e^{-\Gamma(u)}$ in $L^1(\Omega)$. Regarding the second integral in~\eqref{eq:prov2}, we use the Fatou Lemma, which is applicable thanks to Lemma~\ref{lem:bound_hf}. Hence, we obtain that
\begin{equation*}
\into e^{-\Gamma(u)} z \cdot \nabla \varphi \geq \into h(u) e^{-\Gamma(u)} f \varphi,
\end{equation*}
and thus~\eqref{eq:prov3} follows.
\medskip

\textbf{Step 3.} A distributional inequality related to~\eqref{eq:def_dist}.
\smallskip

We aim to show that
\begin{equation}
\label{eq:prov10}
-\operatorname{div}z + |D\Gamma(u)| \leq h(u)f \text{ in } \mathcal{D}'(\Omega).
\end{equation}
Observe that this implies that $z \in \DM_{\rm loc}(\Omega)$. To prove~\eqref{eq:prov10}, we distinguish between two cases: $h(0)<\infty$ and $h(0)=\infty$.

Suppose first that $h(0)<\infty$. Then, given $0\leq \varphi \in C_c^1(\Omega)$, the distributional formulation of~\eqref{eq:PbApprox} yields
\begin{equation}
\label{eq:prov4}
\into z_n \cdot \nabla \varphi + \into \varphi |D\Gamma_n(u_n)| = \into h_n(u_n)f \varphi.
\end{equation}
To pass to the limit in~\eqref{eq:prov4}, we exploit the fact that ${z_n\rightharpoonup z}$ *-weakly in $L^\infty(\Omega)^N$, together with the boundedness of $h$, which follows from $h(0)<\infty$ and assumption~\eqref{eq:hyp_h_inf}. With respect to the second term in~\eqref{eq:prov4}, we rely on the lower semicontinuity of the functional $u \mapsto \into \varphi |Du|$ defined in $\bvo$ with respect to the $L^1$-convergence; recall that $\Gamma_n(u_n)\to \Gamma(u)$ in $L^1(\Omega)$. Therefore, passing to the limit in~\eqref{eq:prov4} we arrive at
\begin{equation}
\label{eq:prov4a}
\into z \cdot \nabla \varphi + \into \varphi |D\Gamma(u)| \leq \into h(u)f \varphi,
\end{equation}
and then~\eqref{eq:prov10} holds.

On the contrary, suppose now that $h(0)=\infty$. In this case, as previously noted, $u>0$ $\ae$ in $\Omega$ (recall that this follows from the fact that $h(u)f\in\luo$ and $f>0$). Let $0\leq \varphi \in C_c^1(\Omega)$ and $\delta>0$, and let us define the nondecreasing function $S_\delta(s) = \min \big\{1,\max \big\{0,\frac{s-\delta}{\delta}\big \}\big\}$ for all $s\geq 0$. Taking $S_\delta(u_n) \varphi \in\bvo\cap\lio$ as test function in~\eqref{eq:app_weak_form} and using the Leibniz rule~\eqref{eq:Prod}, it yields
\begin{equation}
\label{eq:prov9}
\into \varphi (z_n, DS_\delta(u_n)) + \into S_\delta(u_n) z_n \cdot \nabla \varphi + \into g_n(u_n)S_\delta(u_n) \varphi |Du_n| = \into h_n(u_n)f S_\delta(u_n) \varphi.
\end{equation}
Since the function $S_\delta$ is nondecreasing and $S_\delta(u_n) \in\bvo\cap\lio$, an application of Lemma~\ref{lem:Composition2} gives $(z_n, DS_\delta(u_n)) = |DS_\delta(u_n)|$. Then, defining
\[
\Psi_{\delta}(s) = \int_0^s S_\delta(t) g(t) \, \mathrm{d}t,\ \forall s\geq 0, \qquad \text{ and } \qquad \Psi_{\delta,n}(s) = \int_0^s S_\delta(t) g_n(t) \, \mathrm{d}t,\ \forall s\geq 0,
\] 
and recalling that $u_n$ has no jump part,~\eqref{eq:prov9} can be rewritten as
\begin{equation}
\label{eq:prov9a}
\into \varphi |DS_\delta(u_n)| + \into S_\delta(u_n) z_n \cdot \nabla \varphi + \into \varphi |D\Psi_{\delta,n} (u_n)| = \into h_n(u_n)f S_\delta(u_n) \varphi.
\end{equation}
In the following, we pass to the limit in~\eqref{eq:prov9a} when $n\to\infty$. Regarding the first and the third term, we exploit the lower semicontinuity of the functional $u \mapsto \into \varphi |Du|$ defined in $\bvo$ with respect to the $L^1$-covergence; note that $S_\delta(u_n)\to S_\delta(u)$ and $\Psi_{\delta,n}(u_n)\to \Psi_\delta(u)$ in $L^1(\Omega)$. In the second integral, one can take advantage of the fact that ${z_n\rightharpoonup z}$ *-weakly in $L^\infty(\Omega)^N$. With respect to the last term, observe that for every $n\in\N$, the product $h_n(s)S_\delta(s)$ vanishes for $s<\delta$ and, since assumption~\eqref{eq:hyp_h_inf} is in force, the Lebesgue Theorem can be applied. In this way, we obtain that
\begin{equation}
\label{eq:prov9b}
\into \varphi |DS_\delta(u)| + \into S_\delta(u) z \cdot \nabla \varphi + \into \varphi |D\Psi_{\delta} (u)| \leq \into h(u)f S_\delta(u) \varphi.
\end{equation}
Next, we pass to the limit in~\eqref{eq:prov9b} as $\delta\to 0^+$. To this aim, observe that $S_\delta(u) \to \chi_{\{u>0\}}$ in $\luo$ when $\delta$ goes to 0. With respect to the first and the third term of~\eqref{eq:prov9b}, we use again the lower semicontinuity of the functional $u \mapsto \into \varphi |Du|$; note that $\Psi_\delta(u)\to \Gamma(u)$ in $\luo$ when $\delta\to 0^+$. Regarding the last integral, the Lebesgue Theorem can be applied since $h(u)f\in\luo$. Hence, it yields
\begin{equation}
\label{eq:prov9c}
\into \varphi |D\chi_{\{u>0\}}| + \into \chi_{\{u>0\}} z \cdot \nabla \varphi + \into \varphi |D\Gamma (u)| \leq \int_{\{u>0\}} h(u)f \varphi.
\end{equation}
Finally, since $u>0$ $\ae$ in $\Omega$, the inequality~\eqref{eq:prov9c} becomes~\eqref{eq:prov4a} and then~\eqref{eq:prov10} follows.
\medskip

\textbf{Step 4.} The limit $u$ has no jump part.
\smallskip

Here, we use the alternative pairing defined in~\eqref{eq:new_Pairing}. With the notation introduced there, we consider $\beta(s) = -e^{-s}$, $v=\Gamma(u)$ and the pairing $\left( z, D\beta(v)^\# \right)$. From now on, when we write $\left(z, D\left( -e^{-\Gamma(u)} \right)^\# \right)$ we mean the pairing $\left( z, D\beta(v)^\# \right)$.
One deduces that, as measures in $\Omega$, it holds
\begin{equation}
\label{eq:prov11}
\begin{split}
    \left(e^{-\Gamma(u)} \right)^\# |D\Gamma(u)| & \overset{\eqref{eq:prov10}}{\leq} \left(e^{-\Gamma(u)} \right)^\# \operatorname{div}z + e^{-\Gamma(u)} h(u)f \overset{\eqref{eq:prov3}}{\leq} \left(e^{-\Gamma(u)} \right)^\# \operatorname{div}z - \operatorname{div} \left( e^{-\Gamma(u)}z \right)\\
    & \overset{\eqref{eq:new_Pairing}}{=} \left(z, D\left( -e^{-\Gamma(u)} \right)^\# \right) \leq \left|D\left( -e^{-\Gamma(u)} \right) \right| \overset{\eqref{eq:chain_rule_2}}{=} \left(e^{-\Gamma(u)} \right)^\# |D\Gamma(u)|.
\end{split}
\end{equation}
In this way, all the inequalities in~\eqref{eq:prov11} are equalities. In particular, we have that
\begin{equation}
\label{eq:prov12}
    \left(z, D\left( -e^{-\Gamma(u)} \right)^\# \right) = \left|D \left( -e^{-\Gamma(u)} \right) \right| \text{ as measures in } \Omega.
\end{equation}
We can then use Lemma~\ref{lem:no_jump} in~\eqref{eq:prov10} with $\beta(s)=-e^{-s}$ and $v=\Gamma(u)$ to deduce that $D^j \Gamma(u)=0$. Since $\Gamma$ is increasing, we conclude that $D^j u=0$.
\medskip

\textbf{Step 5.} Identification of the vector field $z$ given by~\eqref{eq:def_pair}.
\smallskip

Since $u$ has no jump part and $s\mapsto -e^{-\Gamma(s)}$ is an increasing mapping, then $e^{-\Gamma(u)}$ also has no jump part. Therefore, the identity $\left(- e^{-\Gamma(u)} \right)^\# = -e^{-\Gamma(u)}$ holds $\mathcal{H}^{N-1}$-$\ae$ in $\Omega$ and~\eqref{eq:prov12} can be written as
\[
\left(z, D\left( -e^{-\Gamma(u)} \right) \right) = \left|D \left( -e^{-\Gamma(u)} \right) \right| \text{ as measures in } \Omega.
\]
Given $k>0$, we define the nondecreasing truncation function $\tilde T_k(s) := \min \left\{s, -e^{-\Gamma(k)} \right\}$ for any $s\in\R$. An application of Lemma~\ref{lem:Composition} then yields that
\begin{equation*}
    \left(z, D\tilde T_k\left( -e^{-\Gamma(u)} \right) \right) = \left|D \tilde T_k \left( -e^{-\Gamma(u)} \right) \right| \text{ as measures in } \Omega.
\end{equation*}
Note that $\tilde T_k \left( -e^{-\Gamma(u)} \right) = -e^{-\Gamma(T_k(u))}$. As $s\mapsto -e^{-\Gamma(s)}$ is increasing and $T_k(u)\in \bvo\cap\lio$, using again Lemma~\ref{lem:Composition} (see Remark~\ref{rem:Composition_incr}) we obtain that
\begin{equation}
\label{eq:prov12a}
\left(z, DT_k(u) \right) = \left|DT_k(u) \right| \text{ as measures in } \Omega \text{ for any } k>0.
\end{equation}

\textbf{Step 6.} Distributional formulation~\eqref{eq:def_dist}.
\smallskip

From~\eqref{eq:prov11}, we deduce that
\begin{equation*}
    e^{-\Gamma(u)} (-\operatorname{div} z + |D\Gamma(u)| - h(u)f) = 0 \text{ as measures in } \Omega.
\end{equation*}
Using~\eqref{eq:prov10} and that $e^{-\Gamma(u)}>0$ $\mathcal{H}^{N-1}$-$\ae$ in $\Omega$ because $\Gamma(u) \in \bvo$, one may argue by contradiction to obtain that
\begin{equation}
    \label{eq:prov12bis}
    -\operatorname{div} z + |D\Gamma(u)| = h(u)f \text{ as measures in } \Omega.
\end{equation}
Since $u$ has no jump part, Lemma~\ref{lem:ChainRule} implies that $|D\Gamma(u)| = g(u) |Du|$. From~\eqref{eq:prov12bis} it also follows that $z\in \DM(\Omega)$.
\medskip

\textbf{Step 7.} Boundary condition~\eqref{eq:def_bord}.
\smallskip

Let $\delta >0$. First, we point out that the identity $G_\delta(T_{1+\delta}(s)) = T_1(G_\delta(s))$ holds for any $s\in\R$, and that $G_\delta(T_{1+\delta}(s))=0$ whenever $|s|\leq\delta$. In~\eqref{eq:app_weak_form}, we take $G_\delta(T_{1+\delta}(u_n)) \in \bvo \cap\lio$ as test function to obtain that
\begin{equation}
\label{eq:prov13}
\into |DG_\delta(T_{1+\delta}(u_n))| + \into G_\delta(T_{1+\delta}(u_n)) g_n(u_n) |Du_n| = \into G_\delta(T_{1+\delta}(u_n)) h_n(u_n) f,
\end{equation}
where we have taken into account Lemma~\ref{lem:Composition} and that $u_n\lvert_{\partial\Omega} = 0$. Now, we define the $C^1$ functions
\[
\tilde \Gamma_{\delta}(s) = \int_0^s G_\delta(T_{1+\delta}(t)) g(t) \, \mathrm{d}t,\ \forall s\geq 0, \qquad \text{ and } \qquad \tilde \Gamma_{\delta,n}(s) = \int_0^s G_\delta(T_{1+\delta}(t)) g_n(t) \, \mathrm{d}t,\ \forall s\geq 0.
\]
Observe that $\tilde \Gamma_{\delta}(s), \tilde \Gamma_{\delta,n}(s) > 0$ if and only if $s>\delta$. Using the chain rule,~\eqref{eq:prov13} can be rewritten as
\begin{equation}
\label{eq:prov14}
\into |DG_\delta(T_{1+\delta}(u_n))| + \into |D\tilde\Gamma_{\delta, n}(u_n)| = \into G_\delta(T_{1+\delta}(u_n)) h_n(u_n) f.
\end{equation}
On the left-hand side, we use the lower semicontinuity of the functional $u\mapsto \into |Du| + \intdo |u| \ \dH$ defined in $\bvo$ with respect to the $L^1$-convergence, whereas on the right-hand side we can use the Lebesgue Theorem (note that $G_\delta(T_{1+\delta}(s)) h_n(s)=0$ when $s\leq \delta$ and~\eqref{eq:hyp_h_inf} is in force). Then, passing to the limit in~\eqref{eq:prov14} we obtain that
\begin{equation}
\label{eq:prov15}
\into |DG_\delta(T_{1+\delta}(u))| + \intdo G_\delta(T_{1+\delta}(u)) \ \dH + \into |D\tilde\Gamma_{\delta}(u)| + \intdo \tilde\Gamma_{\delta}(u) \ \dH \leq \into G_\delta(T_{1+\delta}(u)) h(u) f.
\end{equation}

On the other hand, we take $G_\delta(T_{1+\delta}(u))$ as test function in~\eqref{eq:prov12bis} to obtain, after using the Green identity~\eqref{eq:Green} and applying Lemma~\ref{lem:Composition} in~\eqref{eq:prov12a}, that
\begin{equation}
\label{eq:prov16}
\begin{split}
\into G_\delta(T_{1+\delta}(u)) h(u) f &= \into (z, DG_\delta(T_{1+\delta}(u))) - \intdo [G_\delta(T_{1+\delta}(u))z,\nu]\ \dH + \into G_\delta(T_{1+\delta}(u)) g(u) |Du| \\
&= \into |DG_\delta(T_{1+\delta}(u))| - \intdo [G_\delta(T_{1+\delta}(u))z,\nu] \ \dH + \into |D\tilde\Gamma_\delta (u)|.
\end{split}
\end{equation}
We emphasize that, for the identity $|D\tilde\Gamma_\delta (u)| = G_\delta(T_{1+\delta}(u)) g(u) |Du|$ to be true, it is essential that $\Gamma(u)\in \bvo$ (note that $g$ may be unbounded at infinity). Indeed, this implies that $\tilde\Gamma_\delta(u)\in \bvo$ and then Lemma~\ref{lem:ChainRule} can be applied.

Substituting~\eqref{eq:prov16} in~\eqref{eq:prov15}, it follows that
\begin{equation}
\label{prov17}
\intdo G_\delta(T_{1+\delta}(u)) + [G_\delta(T_{1+\delta}(u))z,\nu] \ \dH + \intdo \tilde\Gamma_{\delta}(u) \ \dH \leq 0.
\end{equation}
Since $\big| [G_\delta(T_{1+\delta}(u))z,\nu] \big| \leq G_\delta(T_{1+\delta}(u))$ $\mathcal{H}^{N-1}$-$\ae$ on $\partial\Omega$, then the first integral in~\eqref{prov17} is nonnegative. It follows that the second integral in~\eqref{prov17} is identically zero and thus $\tilde\Gamma_{\delta}(u)=0$ $\mathcal{H}^{N-1}$-$\ae$ on $\partial\Omega$. This yields $\|u\|_{L^\infty(\partial\Omega)} \leq \delta$ and, since $\delta>0$ is arbitrary, we conclude that $u(x)=0$ for $\mathcal{H}^{N-1}$-a.e. $x\in\partial\Omega$.
\end{proof}

\section{Proof of the comparison principle}
\label{sec:Pf_Comp}

The present section is devoted to the demonstration of Theorem~\ref{th:Comp}. To this end, we establish some preliminary results. First, we show that any solution to~\eqref{eq:PbMain} satisfies a certain distributional identity.

\begin{proposition}
\label{prop:dist_comp}
Assume that $g$ verifies~\eqref{eq:hyp_g_int}. If $u\in\bvo$ is a solution to~\eqref{eq:PbMain} with associated vector field $z\in \DM(\Omega)$, then it holds that
\begin{equation*}
-\operatorname{div}\left(e^{-\Gamma(u)}z \right) = h(u)e^{-\Gamma(u)} f \text{ as measures in }\Omega.
\end{equation*}
\end{proposition}

\begin{remark}
Since $h(u)f$ belongs to $\luo$, it follows that $e^{-\Gamma(u)}z\in \DM(\Omega)$.
\end{remark}

\begin{proof}
Let $\varphi\in C^1_c(\Omega)$. Taking the test function $e^{-\Gamma(u)}\varphi \in \bvo\cap\lio$ in~\eqref{eq:def_dist}, we obtain that
\[
-\into e^{-\Gamma(u)}\varphi \operatorname{div}z + \into g(u) e^{-\Gamma(u)}\varphi |Du| = \into h(u) e^{-\Gamma(u)} f \varphi.
\]
Since $e^{-\Gamma(u)} \in \bvo\cap\lio$, the pairing $\big(z, De^{-\Gamma(u)} \big)$ is well-defined, and we can use its definition (see~\eqref{eq:Pairing}) to deduce that
\begin{equation}
\label{eq:Pf_Comp_1a}
\into e^{-\Gamma(u)} z \cdot \nabla \varphi + \into \varphi \left(z, De^{-\Gamma(u)} \right) + \into g(u) e^{-\Gamma(u)}\varphi |Du| = \into h(u) e^{-\Gamma(u)} f \varphi.
\end{equation}
By virtue of Lemma~\ref{lem:Composition2}, since the mapping $s\mapsto e^{-\Gamma(s)}$ is nonincreasing and $e^{-\Gamma(u)} \in \bvo \cap \lio$, it follows that $\big(z, De^{-\Gamma(u)}\big) = - \big|De^{-\Gamma(u)} \big|$. Moreover, the chain rule (Lemma~\ref{lem:ChainRule}) implies that $\big|De^{-\Gamma(u)} \big| = g(u) e^{-\Gamma(u)} |Du|$. Therefore,~\eqref{eq:Pf_Comp_1a} becomes
\[
\into e^{-\Gamma(u)} z \cdot \nabla \varphi = \into h(u) e^{-\Gamma(u)} f \varphi,
\]
and thus our claim follows.
\end{proof}

Next, we prove a technical lemma.

\begin{lemma}
\label{lem:tec_inequality}
Under the assumptions of Theorem~\ref{th:Comp}, the following inequality holds true:
\[
\int_{\{u_1>u_2\}} \left(e^{-\Gamma(u_2)} - e^{-\Gamma(u_1)} \right) (|D\Gamma(u_2)|-|D\Gamma(u_1)|) \geq 0.
\]
\end{lemma}

\begin{proof}
First, we take the test function $\left( e^{-\Gamma(u_2)} - e^{-\Gamma(u_1)} \right)^+ \in \bvo\cap \lio$ in problem~\eqref{eq:PbComp} with ${i=1}$. Since $\left( e^{-\Gamma(u_2)} - e^{-\Gamma(u_1)} \right)^+$ has zero boundary trace and is identically zero if $u_1\leq u_2$ (recall that $\Gamma$ is increasing), applying the Green identity~\eqref{eq:Green} we have that
\begin{equation}
\label{eq:Pf_Comp_1}
\begin{split}
    \int_{\{u_1 > u_2\}} \left( e^{-\Gamma(u_2)} - e^{-\Gamma(u_1)} \right)^+ h(u_1) f_1 =& \int_{\{u_1 > u_2\}} \left(z_1, D\left( e^{-\Gamma(u_2)} - e^{-\Gamma(u_1)} \right) \right) + \int_{\{u_1 > u_2\}} \left( e^{-\Gamma(u_2)} - e^{-\Gamma(u_1)} \right) |D\Gamma(u_1)| \\
    =& \int_{\{u_1 > u_2\}} \left(z_1, De^{-\Gamma(u_2)} \right) + \int_{\{u_1 > u_2\}} \left(z_1, D\left(-e^{-\Gamma(u_1)} \right) \right)\\
    &+ \int_{\{u_1 > u_2\}} \left( e^{-\Gamma(u_2)} - e^{-\Gamma(u_1)} \right) |D\Gamma(u_1)| \\
    =& \int_{\{u_1 > u_2\}} \left(z_1, De^{-\Gamma(u_2)} \right) + \int_{\{u_1 > u_2\}} e^{-\Gamma(u_2)}\left| D\Gamma(u_1) \right|,
\end{split}
\end{equation}
where we have used Lemma~\ref{lem:Composition2} and the chain rule ($D^j u_1=0$) to deduce that
\[
\int_{\{u_1 > u_2\}} \left(z_1, D\left(-e^{-\Gamma(u_1)} \right) \right) = \int_{\{u_1 > u_2\}} \left| D\left(-e^{-\Gamma(u_1)} \right) \right| = \int_{\{u_1 > u_2\}} e^{-\Gamma(u_1)}\left| D\Gamma(u_1) \right|.
\]

On the other hand, taking $\left( e^{-\Gamma(u_2)} - e^{-\Gamma(u_1)} \right)^+$ as test function in problem~\eqref{eq:PbComp} with $i=2$ and making similar computations, we obtain that
\begin{equation}
\label{eq:Pf_Comp_2}
\int_{\{u_1 > u_2\}} \left( e^{-\Gamma(u_2)} - e^{-\Gamma(u_1)} \right)^+ h(u_2) f_2 = -\int_{\{u_1 > u_2\}} \left(z_2, De^{-\Gamma(u_1)} \right) - \int_{\{u_1 > u_2\}} e^{-\Gamma(u_1)}\left| D\Gamma(u_2) \right|
\end{equation}

Since $h$ is nonincreasing by~\eqref{eq:hyp_comp} and $f_1\leq f_2$, we have that $h(u_1)f_1 \leq h(u_2)f_2$ in $\{u_1 > u_2\}$. Then, we can join expressions~\eqref{eq:Pf_Comp_1} and~\eqref{eq:Pf_Comp_2} to deduce that
\begin{equation}
\label{eq:Pf_Comp_3}
\begin{split}
\int_{\{u_1 > u_2\}} e^{-\Gamma(u_2)}\left| D\Gamma(u_1) \right| + \int_{\{u_1 > u_2\}} e^{-\Gamma(u_1)}\left| D\Gamma(u_2) \right|
&\leq -\int_{\{u_1 > u_2\}} \left(z_2, De^{-\Gamma(u_1)} \right) - \int_{\{u_1 > u_2\}} \left(z_1, De^{-\Gamma(u_2)} \right) \\
& \leq \int_{\{u_1 > u_2\}} \left| \left(z_2, De^{-\Gamma(u_1)} \right) \right| + \int_{\{u_1 > u_2\}} \left| \left(z_1, De^{-\Gamma(u_2)} \right) \right| \\
& \leq \int_{\{u_1 > u_2\}} \left| De^{-\Gamma(u_1)} \right| + \int_{\{u_1 > u_2\}} \left| De^{-\Gamma(u_2)} \right| \\
&= \int_{\{u_1 > u_2\}} e^{-\Gamma(u_1)}\left| D\Gamma(u_1) \right| + \int_{\{u_1 > u_2\}} e^{-\Gamma(u_2)}\left| D\Gamma(u_2) \right|,
\end{split}
\end{equation}
where we have used the chain rule and that $\|z_i\|_{\lio^N}\leq 1$ for $i=1,2$. Finally, reordering~\eqref{eq:Pf_Comp_3} we obtain our claim.
\end{proof}

Now, we are in a position to show Theorem~\ref{th:Comp}. As we have already mentioned, our argument is inspired by that of~\cite[Theorem~3.5]{LaSe2}. The proof is as follows.

\begin{proof}[Proof of Theorem~\ref{th:Comp}]

First of all, by Proposition~\ref{prop:dist_comp} we know that
\begin{gather}
\label{eq:Pf_Comp_3a}
-\operatorname{div}\left(e^{-\Gamma(u_1)}z_1 \right) = h(u_1)e^{-\Gamma(u_1)} f_1,\\
\label{eq:Pf_Comp_3b}
-\operatorname{div}\left(e^{-\Gamma(u_2)}z_2 \right) = h(u_2)e^{-\Gamma(u_2)} f_2,
\end{gather}
as measures in $\Omega$. To ease the notation, for any $k>0$ we define the truncated functions
\[
\Gamma_k(s):=\Gamma(T_k(s)),\ \forall s\geq 0.
\]
Given $k>0$, we take $\big(\Gamma_k(u_1) - \Gamma_k(u_2) \big)^+ \in \bvo\cap\lio$ as test function in~\eqref{eq:Pf_Comp_3a}. Since this test function has zero boundary trace, the Green identity~\eqref{eq:Green} gives that
\begin{equation}
\label{eq:Pf_Comp_4}
\into \left( e^{-\Gamma(u_1)} z_1, D \big(\Gamma_k(u_1) - \Gamma_k(u_2) \big)^+ \right) = \into \big(\Gamma_k(u_1) - \Gamma_k(u_2) \big)^+ h(u_1)e^{-\Gamma(u_1)} f_1.
\end{equation}
Performing the same computations in~\eqref{eq:Pf_Comp_3b}, it follows that
\begin{equation}
\label{eq:Pf_Comp_4a}
\into \left( e^{-\Gamma(u_2)} z_2, D \big(\Gamma_k(u_1) - \Gamma_k(u_2) \big)^+ \right) = \into \big(\Gamma_k(u_1) - \Gamma_k(u_2) \big)^+ h(u_2)e^{-\Gamma(u_2)} f_2.
\end{equation}
Subtracting~\eqref{eq:Pf_Comp_4a} to~\eqref{eq:Pf_Comp_4}, we deduce that
\begin{equation}
\label{eq:Pf_Comp_5}
\into \left( e^{-\Gamma(u_1)} z_1 - e^{-\Gamma(u_2)} z_2, D\big(\Gamma_k(u_1) - \Gamma_k(u_2)\big)^+ \right) = \into \big(\Gamma_k(u_1) - \Gamma_k(u_2)\big)^+ \left(h(u_1)e^{-\Gamma(u_1)} f_1 - h(u_2)e^{-\Gamma(u_2)} f_2 \right).
\end{equation}

On the one hand, observe that $s\mapsto h(s)e^{-\Gamma(s)}$ is nonincreasing; this follows from the fact that $h$ is nonincreasing by~\eqref{eq:hyp_comp} and $\Gamma$ is increasing. Using also that $f_1\leq f_2$ and that $\big(\Gamma_k(u_1) - \Gamma_k(u_2)\big)^+ = 0$ if $u_1\leq u_2$, we have
\begin{equation}
\label{eq:Pf_Comp_6}
\big(\Gamma_k(u_1) - \Gamma_k(u_2)\big)^+ \left(h(u_1)e^{-\Gamma(u_1)} f_1 - h(u_2)e^{-\Gamma(u_2)} f_2 \right) \leq 0 \ \ae \text{ in } \Omega.
\end{equation}

On the other hand, the limit of the left-hand side of~\eqref{eq:Pf_Comp_5} as $k$ diverges is nonnegative. This can be shown as follows. To simplify the notation, we denote $A_k:=\{\Gamma_k(u_1)> \Gamma_k(u_2)\}$. Taking into account that $\|z_i\|_{\lio^N} \leq 1$ and $(z_i,D\Gamma_k(u_i)) = |D\Gamma_k(u_i)|$ for $i=1,2$ thanks to Lemma~\ref{lem:Composition}, we deduce that
\begin{equation}
\label{eq:Pf_Comp_7}
\begin{split}
\into & \left( e^{-\Gamma(u_1)} z_1 - e^{-\Gamma(u_2)} z_2, D\big(\Gamma_k(u_1) - \Gamma_k(u_2)\big)^+ \right)\\
=& \int_{A_k} \left( e^{-\Gamma(u_1)} z_1 - e^{-\Gamma(u_2)} z_2, D\big(\Gamma_k(u_1) - \Gamma_k(u_2)\big) \right) \\
=& \int_{A_k} e^{-\Gamma(u_1)} |D\Gamma_k(u_1)| + \int_{A_k} e^{-\Gamma(u_2)} |D\Gamma_k(u_2)| - \int_{A_k} e^{-\Gamma(u_1)} (z_1, D\Gamma_k(u_2)) - \int_{A_k} e^{-\Gamma(u_2)} (z_2, D\Gamma_k(u_1)) \\
\geq & \int_{A_k} e^{-\Gamma(u_1)} |D\Gamma_k(u_1)| + \int_{A_k} e^{-\Gamma(u_2)} |D\Gamma_k(u_2)| - \int_{A_k} e^{-\Gamma(u_1)} |D\Gamma_k(u_2)| - \int_{A_k} e^{-\Gamma(u_2)} |D\Gamma_k(u_1)| \\
=& \int_{A_k\cap \{u_2<k\}} \left( e^{-\Gamma(u_2)} - e^{-\Gamma(u_1)} \right) |D\Gamma(u_2)| - \int_{A_k \cap \{u_1<k\}} \left( e^{-\Gamma(u_2)} - e^{-\Gamma(u_1)} \right) |D\Gamma(u_1)|.
\end{split}
\end{equation}
We point out that in~\eqref{eq:Pf_Comp_7} we have used the property $(wz,Dv) = w(z,Dv)$, which holds whenever $z\in \DM(\Omega)$ and $v,w\in\bvo\cap\lio$ are such that $D^j v = D^j w = 0$ (see~\cite[Proposition~2.3]{MaSe}).

Taking limits in~\eqref{eq:Pf_Comp_7} as $k\to \infty$, and using the Monotone Convergence Theorem on the right-hand side, we obtain that
\begin{equation}
\label{eq:Pf_Comp_8}
\liminf_{k\to \infty} \into \left( e^{-\Gamma(u_1)} z_1 - e^{-\Gamma(u_2)} z_2, D\big(\Gamma_k(u_1) - \Gamma_k(u_2)\big)^+ \right) \geq \int_{\{u_1>u_2\}} \left(e^{-\Gamma(u_2)} - e^{-\Gamma(u_1)} \right) (|D\Gamma(u_2)|-|D\Gamma(u_1)|) \geq 0,
\end{equation}
where the last inequality follows by Lemma~\ref{lem:tec_inequality}.

Now, taking into account the relation~\eqref{eq:Pf_Comp_5} and the inequalities~\eqref{eq:Pf_Comp_6} and~\eqref{eq:Pf_Comp_8}, we deduce that
\begin{equation}
\label{eq:Pf_Comp_9}
\lim_{k\to \infty} \into \big(\Gamma_k(u_1) - \Gamma_k(u_2)\big)^+ \left(h(u_1)e^{-\Gamma(u_1)} f_1 - h(u_2)e^{-\Gamma(u_2)} f_2 \right) = 0.
\end{equation}
Using again the Monotone Convergence Theorem in~\eqref{eq:Pf_Comp_9}, we conclude that
\begin{equation}
\label{eq:Pf_Comp_10}
\into \big(\Gamma(u_1) - \Gamma(u_2)\big)^+ \left(h(u_1)e^{-\Gamma(u_1)} f_1 - h(u_2)e^{-\Gamma(u_2)} f_2 \right) = 0.
\end{equation}

Suppose, by contradiction, that there exists a subset $\omega \subset \Omega$ with positive measure such that $u_1>u_2$ in $\omega$. Since $h(u_1)e^{-\Gamma(u_1)} f_1 \leq h(u_2)e^{-\Gamma(u_2)} f_2$ whenever $u_1>u_2$ (recall that $h$ is nonincreasing by~\eqref{eq:hyp_comp}), then equality~\eqref{eq:Pf_Comp_10} yields that
\begin{equation}
\label{eq:Pf_Comp_11}
h(u_1)e^{-\Gamma(u_1)} f_1 = h(u_2)e^{-\Gamma(u_2)} f_2\  \text{ in } \omega.
\end{equation}
At this point, we distinguish two cases: the first, in which $h$ is positive on $(0,\infty)$, and the second, in which $h$ vanishes at some $s_0>0$.

If $h(s)>0$ for all $s>0$, then the mapping $s\mapsto h(s)e^{-\Gamma(s)}$ is decreasing because $h$ is nonincreasing by~\eqref{eq:hyp_comp} and $\Gamma$ is increasing. Using also that $f_2>0$, one deduces
\begin{equation}
\label{eq:Pf_Comp_12}
h(u_1)e^{-\Gamma(u_1)} f_2 < h(u_2)e^{-\Gamma(u_2)} f_2\  \text{ in } \omega.
\end{equation}
Joining~\eqref{eq:Pf_Comp_11} and~\eqref{eq:Pf_Comp_12}, we obtain that $f_1>f_2$ in $\omega$, a contradiction.

Finally, if $h(s_0)=0$ for some $s_0>0$, the monotonicity of $h$ ensures that $h(s)=0$ for all $s\geq s_0$. In this case, both $u_1$ and $u_2$ belong to $\lio$ and are bounded by $s_0$, as shown in Proposition~\ref{prop:Bound_h}. Since the mapping $s\mapsto h(s)e^{-\Gamma(s)}$ is decreasing whenever $s\in (0,s_0]$, we may proceed as before to arrive at the same contradiction.

Therefore, we conclude that $u_1\leq u_2$ $\ae$ in $\Omega$.
\end{proof}

\section{$L^\infty$-regularity of solutions}
\label{sec:Bound}

\subsection{Existence of bounded solution when $h$ vanishes at some point}
\label{sec:Bound_h=0}

The existence results stated in Section~\ref{sec:Main} are valid if $h$ touches the axis at some point. However, in this setting, stronger results can be established. We first prove a preliminary result showing that if $h$ remains equal to zero from a certain point onwards, then any solution to~\eqref{eq:PbMain} is bounded.

\begin{proposition}
\label{prop:Bound_h}
Let $0<f\in L^1(\Omega)$. Assume that $g$ verifies~\eqref{eq:hyp_g_int} and that there exists some $s_0>0$ such that $h(s)=0$ for all $s\geq s_0$. Then any solution $u\in \bvo$ of~\eqref{eq:PbMain} belongs to $\lio$ and satisfies $\|u\|_\lio \leq s_0$.
\end{proposition}

\begin{proof}
Let $\ell >s_0>0$. Observe that $h(s)G_{s_0}(T_\ell (s))=0$ for all $s\geq 0$. In this way, taking $G_{s_0}(T_\ell (u)) \in \bvo\cap \lio$ as test function in~\eqref{eq:def_dist} and using the Green identity~\eqref{eq:Green}, one obtains that
\[
\into (z, DG_{s_0}(T_\ell (u))) + \into g(u) G_{s_0}(T_\ell (u)) |Du| = 0,
\]
where we have taken into account that $u\lvert_{\partial\Omega}=0$. Since the second integral is nonnegative, an application of Lemma~\ref{lem:Composition} yields
\begin{equation*}
\into |DG_{s_0}(T_\ell (u))|=0.
\end{equation*}
In this way, $G_{s_0}(T_\ell (u))\equiv 0$ and, since $\ell >s_0$, it follows that $\|u\|_\lio \leq s_0$.
\end{proof}

The previous result can be combined with Theorem~\ref{th:Exist2} to ensure the existence of a bounded solution to~\eqref{eq:PbMain} if $h$ vanishes at some point, independently of the behaviour of $g$ or $h$ at infinity.

\begin{theorem}
Let $0<f\in L^1(\Omega)$. Assume that $g$ verifies~\eqref{eq:hyp_g_int} and that there exists some $s_0>0$ such that $h(s_0)=0$. Then problem~\eqref{eq:PbMain} has a nonnegative solution $u\in BV(\Omega)\cap \lio$ satisfying that $\|u\|_\lio \leq s_0$.
\end{theorem}

\begin{proof}
For each $s\geq 0$, we define the continuous function
\[
\tilde h(s):=
\begin{cases}
h(s) & \text{ if } s< s_0,\\
0 & \text{ if } s\geq s_0,
\end{cases}
\]
and we consider the problem
\begin{equation}
	\label{eq:PbZero}
	\begin{cases}
		-\Delta_1 u + g(u)|Du| = \tilde h(u)f & \text{in }\Omega,\\
		u=0 & \text{on } \partial\Omega.
	\end{cases}
\end{equation}
Since the assumptions of Theorem~\ref{th:Exist2} are verified, problem~\eqref{eq:PbZero} has a solution $u\in\bvo$. By Proposition~\ref{prop:Bound_h}, $u$ belongs to $\lio$ and $\|u\|_\lio \leq s_0$. Since $\tilde h(s) = h(s)$ when $s\leq s_0$, $u$ also solves~\eqref{eq:PbMain}.
\end{proof}

\subsection{Bounded solutions for more regular data}

We show that any solution to~\eqref{eq:PbMain} belongs to $\lio$ provided that $f$ is sufficiently integrable. This boundedness result, stated below, has been shown when $g\equiv h\equiv 1$ in~\cite{MaSe, LaSe1}.

\begin{proposition}
Let $0<f\in L^m(\Omega)$ with $m\geq 1$. Assume that $g$ verifies~\eqref{eq:hyp_g_int} and that $h$ satisfies~\eqref{eq:hyp_h_inf}. Then any solution $u\in \bvo$ of~\eqref{eq:PbMain} belongs to $\lio$ if one of the following cases occurs:
\begin{enumerate}
    \item[(i)] $m>N$,
    \item[(ii)] $m=N$ and $\|f\|_\lno < (\mathcal{S}_1 h(\infty))^{-1}$, where $\mathcal{S}_1$ is defined in~\eqref{eq:sob_embed}.
\end{enumerate}
\end{proposition}

\begin{remark}
We stress that this is not an existence result, but rather an \emph{a posteriori} estimate that can be combined with Theorems~\ref{th:Exist} and~\ref{th:Exist2}. Indeed, the same proof remains valid if $g\equiv 0$ and $h\equiv 1$, a case where a strong nonexistence phenomenon occurs even for (large) constant data $f$ (see~\cite{MeSeTr}).
\end{remark}

\begin{remark}
The second case can be analogously extended to the setting where $f$ belongs to the Marcinkiewicz space $L^{N,\infty}(\Omega)$. Recall that $L^{N,\infty}(\Omega)$ is strictly larger than $L^N(\Omega)$, but is contained in $L^q(\Omega)$ for any $q\in[1,N)$. In this framework, unbounded solutions may arise if $\|f\|_{L^{N,\infty}(\Omega)}$ is sufficiently large (see~\cite[Section~5]{LaSe1}).
\end{remark}

\begin{proof}
Let $\ell >k>0$. We take $G_k(T_\ell (u)) \in \bvo\cap \lio$ as test function in~\eqref{eq:def_dist}, and we use the Green identity~\eqref{eq:Green} to obtain that
\[
\into (z, DG_k(T_\ell (u))) + \into g(u) G_k(T_\ell (u)) |Du| = \into h(u) G_k(T_\ell (u)) f,
\]
where we have taken into account that $u\lvert_{\partial\Omega}=0$. Using Lemma~\ref{lem:Composition} and dropping the second integral, it follows that
\begin{equation*}
\into |DG_k(T_\ell (u))| \leq \into h(u) G_k(T_\ell (u)) f.
\end{equation*}
Since~\eqref{eq:hyp_h_inf} holds, then $\sup_{s>k} h(s)$ is finite for any $k>0$. As $G_k(s)=0$ when $s\leq k$, from the previous we obtain 
\begin{equation*}
\into |DG_k(T_\ell (u))| \leq \sup_{s>k}h(s)\into G_k(T_\ell (u)) f.
\end{equation*}
Now, for each $k>0$ we define the sets $A_k := \{u > k\}$. Using H\"older's inequality and the Sobolev inequality~\eqref{eq:sob_embed}, and taking into account that $u\lvert_{\partial\Omega} = 0$, we deduce that
\begin{equation}
\label{eq:Pf_Bound_1}
\begin{split}
\into |DG_k(T_\ell (u))| &\leq \sup_{s>k}h(s) \|f\|_{L^m(\Omega)} |A_k|^{\frac{1}{N}-\frac{1}{m}} \|G_k(T_\ell (u))\|_{L^{\frac{N}{N-1}}(\Omega)}\\
&\leq \mathcal{S}_1 \sup_{s>k}h(s) \|f\|_{L^m(\Omega)} |A_k|^{\frac{1}{N}-\frac{1}{m}} \into |DG_k(T_\ell (u))|.
\end{split}
\end{equation}

We claim that there exists some $k_0>0$ such that
\begin{equation}
\label{eq:Pf_Bound_2}
\mathcal{S}_1 \sup_{s>k_0}h(s) \|f\|_{L^m(\Omega)} |A_{k_0}|^{\frac{1}{N}-\frac{1}{m}} < 1.
\end{equation}
Indeed, if $m>N$, the fact that $u\in\luo$ implies that $|A_{k}|^{\frac{1}{N}-\frac{1}{m}} \to 0$ as $k\to\infty$, from which~\eqref{eq:Pf_Bound_2} follows. On the other hand, if $m=N$, the assumption $\|f\|_\lno < (\mathcal{S}_1 h(\infty))^{-1}$ ensures the existence of some $k_0>0$ such that~\eqref{eq:Pf_Bound_2} holds.

Therefore, from~\eqref{eq:Pf_Bound_1} and~\eqref{eq:Pf_Bound_2} we obtain that
\[
\into |DG_{k_0}(T_\ell (u))| = 0,
\]
which implies that $G_{k_0}(T_\ell (u)) \equiv 0$. Since $\ell>k_0$, it follows that $\|u\|_{\lio} \leq k_0$.
\end{proof}

\section{Additional comments}

\subsection{A weaker concept of solution}
\label{sec:Weak_sol}
The 1-Laplacian operator is formally invariant under increasing transformations. Indeed, observe that
\[
\frac{D\Gamma(u)}{|D\Gamma(u)|} = \frac{Du}{|Du|},
\]
provided that $\Gamma'>0$ and $u$ is regular enough. This property allows us to reformulate~\eqref{eq:PbMain} with a ``normalized'' gradient term. 

If $u$ is a solution to~\eqref{eq:PbMain} and one performs the change  of variables $v=\Gamma(u)$ (recall that $\Gamma$ is increasing), then $v$ formally solves
\begin{equation}
	\label{eq:PbCV2}
	\begin{cases}
		-\Delta_1 v + |Dv| = h\left(\Gamma^{-1}(v)\right)f & \text{in }\Omega,\\
		v=0 & \text{on } \partial\Omega.
	\end{cases}
\end{equation}
The next lemma formalizes this observation.

\begin{proposition}
\label{prop:sol_equiv}
Assume that $g$ verifies~\eqref{eq:hyp_g_int}. If $u\in \bvo$ is a solution to~\eqref{eq:PbMain} with associated vector field $z\in \DM(\Omega)$, then $v:=\Gamma(u) \in \bvo$ is a solution to~\eqref{eq:PbCV2} with the same associated vector field $z$.
\end{proposition}

\begin{proof}
We need to prove that $v$ verifies all the requirements of Definition~\ref{def:sol}. First, since~\eqref{eq:hyp_g_int} is assumed, then $\Gamma(0)=0$ and thus $v$ verifies the boundary condition~\eqref{eq:def_bord}. Moreover, it is clear that $v$ satisfies the weak formulation of~\eqref{eq:PbCV2} with the vector field $z$. Finally, as for any $\ell>0$ we have
\[
(z, DT_\ell(u)) = |DT_\ell(u)| \text{ as measures in } \Omega,
\]
an application of Lemma~\ref{lem:Composition} gives that
\[
(z, D\Gamma(T_\ell(u))) = |D\Gamma(T_\ell(u))| \text{ as measures in } \Omega.
\]
By choosing $\ell = \Gamma(k)$ for a given $k>0$, we observe that $\Gamma(T_\ell(u)) = T_k(\Gamma(u)) = T_k(v)$. This confirms that $v$ satisfies the condition~\eqref{eq:def_pair}, thus concluding that $v$ is a solution to~\eqref{eq:PbCV2}.
\end{proof}

One may wonder if the converse to Proposition~\ref{prop:sol_equiv} also holds; that is, whether $u:=\Gamma^{-1}(v)$ is a solution to~\eqref{eq:PbMain} provided that $v\in\bvo$ is a solution to~\eqref{eq:PbCV2}. The answer to this question is negative. The reason is that, in general, one cannot guarantee that $u$ belongs to $\bvo$.

This phenomenon is perfectly illustrated in~\cite[Example~7.4]{LaSe1}. In this example, the authors consider $N>1$, $\Omega=B_1(0)$, $h\equiv 1$, $g(s)=\frac{1}{1+s}$ and $f(x)=\frac{\lambda}{|x|}$ with $\lambda>2(N-1)$, and they show that 
\[
v(x)=(N-1-\lambda) \log|x|\in\bvo
\]
solves~\eqref{eq:PbCV2} with associated vector field $z(x)=-\frac{x}{|x|}$, yet $u:=\Gamma^{-1}(v)$ does not belong to $\bvo$.

The above considerations lead to the following weaker concept of solution.

\begin{defin}
Let $0<f\in\luo$. We say that a nonnegative function $u$ is a \textit{weak solution} to problem~\eqref{eq:PbMain} if $u(x)<\infty$ $\ae$ in $\Omega$ and $v:= \Gamma(u)$ solves (in the sense of Definition~\ref{def:sol}) problem~\eqref{eq:PbCV2}.
\end{defin}

As shown in Proposition~\ref{prop:sol_equiv}, any solution to~\eqref{eq:PbMain} is also a weak solution. Conversely, by reasoning analogously to Proposition~\ref{prop:sol_equiv}, it can be proved that the reciprocal also holds if the weak solution belongs to $\bvo$.

\begin{proposition}
Assume that $g$ verifies~\eqref{eq:hyp_g_int}. If $u$ is a weak solution to~\eqref{eq:PbMain} and $u\in \bvo$, then $u$ is a solution to~\eqref{eq:PbMain}.
\end{proposition}

Observe that problem~\eqref{eq:PbCV2} presents a structural advantage with respect to~\eqref{eq:PbMain}: the gradient term is multiplied by 1, and, clearly, this coefficient does not degenerate at infinity. Hence, Theorem~\ref{th:Exist} can be applied to ensure the existence of a solution $v\in\bvo$ to~\eqref{eq:PbCV2}. Consequently, if $\Gamma^{-1}$ is defined in $[0,\infty)$, then $u:=\Gamma^{-1}(v)$ is finite $\ae$, and thus $u$ is a weak solution to~\eqref{eq:PbMain}. All of these considerations are summarized in the following existence theorem.

\begin{theorem}
\label{th:Exist3}
Let $0<f\in L^1(\Omega)$. Assume that $g$ verifies~\eqref{eq:hyp_g_int} and that $h$ satisfies~\eqref{eq:hyp_h_inf}. If $g\notin L^1([1,\infty))$, then problem~\eqref{eq:PbMain} has a weak solution $u$.
\end{theorem}

The assumption $g\notin L^1([1,\infty))$ is made to ensure that $\Gamma^{-1}$ is defined on $[0,\infty)$. This result should be compared with Theorem~\ref{th:Exist}: note that the assumptions of Theorem~\ref{th:Exist3} are more relaxed than those of Theorem~\ref{th:Exist}. We also point out that a similar result for the case $h\equiv 1$ and $f\in L^{N,\infty}(\Omega)$ has been established in~\cite[Theorem~7.3]{LaSe1}.

Finally, we underline that if $g\in L^1([1,\infty))$, problem~\eqref{eq:PbMain} may not have a weak solution. The heuristic reason is that, in that case, the domain of $\Gamma^{-1}$ is $[0, L)$, where $L:=\lim_{s\to\infty} \Gamma(s)<\infty$. As a consequence, the range of a solution $v$ to~\eqref{eq:PbCV2} may exceed this interval, making it impossible to define $u:=\Gamma^{-1}(v)$. A further discussion on this phenomenon can be found in~\cite[Section~8.1]{LaSe1}.

\subsection{Extensions}
\label{sec:Extensions}

For the sake of exposition, we have assumed conditions on $g$, $h$ and $f$ that are more restrictive than actually necessary. In this final section, we outline some possible generalizations of the results established throughout the article.

Firstly, all the results and proofs presented in this work remain valid if the condition $g>0$ is relaxed to $g\geq 0$, provided that $\Gamma$ is increasing. This allows $g$ to vanish at isolated points.

Secondly, if $h$ is nonsingular (i.e. $h(0)<\infty$), the concept of solution (Definition~\ref{def:sol}) remains unchanged when $f$ is merely nonnegative. In this case, all the results remain valid, with the only difference arising at the end of the proof of the comparison principle (Theorem~\ref{th:Comp}). In that proof, the condition $f_2>0$ plays a key role. If $f_2\geq 0$, one can instead follow from~\eqref{eq:Pf_Comp_10} the arguments of~\cite[Theorem~3.5]{LaSe2} to conclude the proof.

The case where $h$ is singular (i.e. $h(0)=\infty$) and $f$ is nonnegative is more delicate. As was first pointed out in~\cite{DecGiSe}, the appropriate concept of solution (where the solutions to the $p$-Laplacian problems converge as $p\to 1^+$) requires the characteristic function $\chi_{\{u>0\}}$ in its formulation. These solutions also satisfy a variational formulation that is reminiscent of the concept of renormalized solutions.

The mentioned notion of solution is formulated as follows (see~\cite{BalOP}).

\begin{defin}
\label{def:sol_nonneg}
Let $0\leq f\in\luo$ and assume that $h(0)=\infty$. A nonnegative function $u\in BV(\Omega)$ is a \textit{solution} to problem~\eqref{eq:PbMain} if $D^j u=0$, $g(u)\in L^1(\Omega,|Du|)$, $h(u)f \in L^1(\Omega)$, $\chi_{\{u>0\}}\in BV_{\rm loc}(\Omega)$, and if there exists $z\in \DM_{\rm loc}(\Omega)$ with $\|z\|_{L^\infty(\Omega)^N}\leq 1$ such that
\begin{gather} 
    \label{eq:NN_dist}
    -\chi_{\{u>0\}}^* \operatorname{div}z + g(u)|Du| = h(u)f \text{ as measures in } \Omega,\\
    \label{eq:NN_pair}
    (z,DT_k(u))=|DT_k(u)| \text{ as measures in } \Omega \text{ for any } k>0,\\
    \label{eq:NN_bound}
    u(x)=0 \text{ for  $\mathcal{H}^{N-1}$-a.e. } x \in \partial\Omega.		
\end{gather}
\end{defin}

\begin{remark}
It must be specified that $h(0)f(x)$ is defined as 0 in the set where $f(x)=0$. In contrast, we understand that $h(0)f(x)=\infty$ when $f(x)>0$.
\end{remark}

Observe that the above definition coincides with the concept of solution introduced in Definition~\ref{def:sol} when $f>0$. Indeed, in this case, any solution in the sense of Definition~\ref{def:sol_nonneg} must be strictly positive, as otherwise the product $h(u)f$ cannot belong to $\luo$. Thus, under the condition $f>0$, the Definition~\ref{def:sol} is recovered.

In this setting, an existence result analogous to Theorems~\ref{th:Exist} and~\ref{th:Exist2} can be shown. Although the fundamental ideas remain the same, several technical difficulties arise. The proof is as follows.

\begin{theorem}
Let $0\leq f\in L^1(\Omega)$ and suppose that $h(0)=\infty$. Under the assumptions of Theorem~\ref{th:Exist} or Theorem~\ref{th:Exist2} (excluding the requirement $f>0$), problem~\eqref{eq:PbMain} has a nonnegative solution $u\in BV(\Omega)$ in the sense of Definition~\ref{def:sol_nonneg}.
\end{theorem}

\begin{proof}
We set $f_n(x) = \max\left\{\frac{1}{n}, f(x)\right\}$ for all $x\in\Omega$, and we define $g_n$ and $h_n$ as in Section~\ref{sec:Approx}. We consider the approximated problems
\begin{equation}
	\label{eq:PbApprox_NN}
	\begin{cases}
		-\Delta_1 u_n + g_n(u_n)|Du_n| = h_n(u_n)f_n & \text{in }\Omega,\\
		u_n=0 & \text{on } \partial\Omega.
	\end{cases}
\end{equation}
Since $f_n$ is strictly positive in $\Omega$, the existence of a solution $u_n\in \bvo$ to~\eqref{eq:PbApprox} (in the sense of Definition~\ref{def:sol}) is guaranteed by Theorem~\ref{th:Exist}. We denote by $z_n\in\DM(\Omega)$ the vector field associated to $u_n$.

All the estimates of Section~\ref{sec:Approx} can be adapted in a simple manner to this setting. Therefore, one finds that both sequences $u_n$ and $\Gamma_n(u_n)$ are bounded in $\bvo$, and that the sequence $h_n(u_n)f_n$ is bounded in $\luo$. As a result, there exists $u\in\bvo$ and a subsequence of $u_n$ (not relabelled) such that
\[
u_n\to u \text{ in } \luo.
\]
Since $\Gamma_n(u_n)$ is bounded in $\bvo$, it follows that $\Gamma(u)\in\bvo$ and $\Gamma_n(u_n) \to \Gamma(u)$ in $\luo$. Moreover, as $h_n(u_n)f_n$ is bounded in $\luo$, the Fatou Lemma gives that $h(u)f\in\luo$. Consequently, the inclusion $\{u=0\}\subseteq \{f=0\}$ holds up to a set of measure zero, which ensures that
\begin{equation}
\label{eq:Pf_NN_1b}
\int_{\{u>0\}} h(u)f  = \into h(u)f.
\end{equation}

Next, we follow the same steps as in the proof of Theorems~\ref{th:Exist} and~\ref{th:Exist2}, highlighting only the real differences.
\medskip

\textbf{Step 1.} Since $\|z_n\|_{L^\infty(\Omega)^N}\leq 1$, one finds that, up to a subsequence, $z_n\rightharpoonup z$ *-weakly in $L^\infty(\Omega)^N$ for some $z\in {L^\infty(\Omega)^N}$ with $\|z\|_{L^\infty(\Omega)^N}\leq 1$. Furthermore, $z$ belongs to $\DM_{\rm loc}(\Omega)$. Indeed, given $\varphi \in C^1_c(\Omega)$, from the distributional formulation of~\eqref{eq:PbApprox_NN} we deduce the existence of some $C>0$ such that
\[
\left| \into z_n \cdot \nabla \varphi \right| \leq \left|\into \varphi |D\Gamma_n(u_n) |\right| + \left|\into h_n(u_n) f_n \varphi \right| \leq C \|\varphi \|_\lio,
\]
where we have used that $\Gamma_n(u_n)$ is bounded in $\bvo$ and $h_n(u_n)f_n$ is bounded in $\luo$. Passing to the limit using the weak-* convergence of $z_n$, we conclude that $\operatorname{div}z$ is a locally finite Radon measure.

\medskip

\textbf{Step 2.} Arguing as in the proof of Theorems~\ref{th:Exist} and~\ref{th:Exist2}, one shows that
\begin{equation}
\label{eq:Pf_NN_1}
-\operatorname{div} \left( e^{-\Gamma(u)}z \right) \geq h(u) e^{-\Gamma(u)} f \text{ in } \mathcal{D}'(\Omega).
\end{equation}
Note that from~\eqref{eq:Pf_NN_1}, it follows that $e^{-\Gamma(u)}z\in\DM_{\rm loc}(\Omega)$. 
\medskip

\textbf{Step 3.} In this setting, we aim to show that
\begin{equation}
\label{eq:Pf_NN_1a}
-\chi_{\{u>0\}}^* \operatorname{div} z + |D\Gamma(u)| \leq h(u)f \text{ in } \mathcal{D}'(\Omega).
\end{equation}
Let $0\leq \varphi\in C_c^1(\Omega)$. Reasoning as in Step 3 (case $h(0)=\infty$) of the proof of Theorems~\ref{th:Exist} and~\ref{th:Exist2}, one obtains (see~\eqref{eq:prov9c}) that
\begin{equation}
\label{eq:Pf_NN_2}
\into \varphi |D\chi_{\{u>0\}}| + \into \chi_{\{u>0\}} z \cdot \nabla \varphi + \into \varphi |D\Gamma (u)| \leq \int_{\{u>0\}} h(u)f \varphi.
\end{equation}
First, observe that from~\eqref{eq:Pf_NN_2} one deduces that $\chi_{\{u>0\}} \in BV_{\rm loc}(\Omega)$ and $\chi_{\{u>0\}}z\in\DM_{\rm loc}(\Omega)$. Therefore, the pairing $(z,D\chi_{\{u>0\}})$ is well-defined. Applying~\eqref{eq:AbsCont} in~\eqref{eq:Pf_NN_2} and taking~\eqref{eq:Pf_NN_1b} into account, it follows that
\begin{equation}
\label{eq:Pf_NN_3}
\into \varphi (z,D\chi_{\{u>0\}}) + \into \chi_{\{u>0\}} z \cdot \nabla \varphi + \into \varphi |D\Gamma (u)| \leq \into h(u)f \varphi.
\end{equation}
Using the definition of the pairing (see~\eqref{eq:Pairing}), the inequality~\eqref{eq:Pf_NN_3} becomes
\begin{equation*}
-\into \varphi \chi_{\{u>0\}}^* \operatorname{div}z + \into \varphi |D\Gamma (u)| \leq \into h(u)f \varphi
\end{equation*}
and then~\eqref{eq:Pf_NN_1a} follows.
\medskip

\textbf{Step 4.} We follow the arguments presented in~\cite{BalOP}. Using the same notation as in the proof of Theorems~\ref{th:Exist} and~\ref{th:Exist2}, one deduces that, as measures in $\Omega$, it holds
\begin{equation}
\label{eq:Pf_NN_4}
\begin{split}
    \left(e^{-\Gamma(u)} \right)^\# \chi_{\{u>0\}}^* |D\Gamma(u)| & \overset{\eqref{eq:Pf_NN_1a}}{\leq} \left(e^{-\Gamma(u)} \right)^\# \chi_{\{u>0\}}^* \operatorname{div}z + e^{-\Gamma(u)} h(u)f \chi_{\{u>0\}}\\
    &\overset{\eqref{eq:Pf_NN_1}}{\leq} \left(e^{-\Gamma(u)} \right)^\# \chi_{\{u>0\}}^* \operatorname{div}z - \chi_{\{u>0\}}^* \operatorname{div} \left( e^{-\Gamma(u)}z \right)\\
    & \overset{\eqref{eq:new_Pairing}}{=} \chi_{\{u>0\}}^* \left(z, D\left( -e^{-\Gamma(u)} \right)^\# \right) \leq \chi_{\{u>0\}}^* \left|D\left( -e^{-\Gamma(u)} \right) \right| \overset{\eqref{eq:chain_rule_2}}{=} \chi_{\{u>0\}}^* \left(e^{-\Gamma(u)} \right)^\# |D\Gamma(u)|.
\end{split}
\end{equation}
Therefore, all the inequalities in~\eqref{eq:Pf_NN_4} are equalities. In particular, we have that
\begin{equation}
\label{eq:Pf_NN_5}
    \chi_{\{u>0\}}^* \left(z, D\left( -e^{-\Gamma(u)} \right)^\# \right) = \chi_{\{u>0\}}^* \left|D \left( -e^{-\Gamma(u)} \right) \right| \text{ as measures in } \Omega.
\end{equation}
We can then use Lemma~\ref{lem:no_jump} (see Remark~\ref{rem:no_jump}) in~\eqref{eq:Pf_NN_1a} with $\beta(s)=-e^{-s}$ and $v=\Gamma(u)$ to deduce that $D^j \Gamma(u)=0$. Since $\Gamma$ is increasing, we conclude that $D^j u=0$.
\medskip

\textbf{Step 5.} Since $u$ has no jump part, then $e^{-\Gamma(u)}$ also has no jump part. Therefore, the identity $\left(- e^{-\Gamma(u)} \right)^\# = -e^{-\Gamma(u)}$ holds $\mathcal{H}^{N-1}$-$\ae$ in $\Omega$.

Moreover, $e^{-\Gamma(u)}\in \bvo$ since $\Gamma(u)\in \bvo$. Noting that $\{u=0\}$ coincides with the set $\left\{e^{-\Gamma(u)} = 1\right\}$, it follows from~\cite[Proposition~3.92]{AFP} that $De^{-\Gamma(u)} \resmes \{u=0\} = 0$. By~\eqref{eq:AbsCont}, we also have that $\left(z, D\left(-e^{-\Gamma(u)}\right) \right)$ vanishes on the set $\{u=0\}$. Hence,~\eqref{eq:Pf_NN_5} can be rewritten as
\begin{equation*}
    \left(z, D\left( -e^{-\Gamma(u)} \right) \right) = \left|D \left( -e^{-\Gamma(u)} \right) \right| \text{ as measures in } \Omega.
\end{equation*}
From this point, one can follow the same procedure as in the proof of Theorems~\ref{th:Exist} and~\ref{th:Exist2} to show that~\eqref{eq:NN_pair} holds.
\medskip

\textbf{Step 6.} From~\eqref{eq:Pf_NN_5}, we deduce that
\begin{equation}
\label{eq:Pf_NN_6}
    e^{-\Gamma(u)} (-\chi_{\{u>0\}}^* \operatorname{div} z + \chi_{\{u>0\}} |D\Gamma(u)| - h(u)f \chi_{\{u>0\}}) = 0 \text{ as measures in } \Omega.
\end{equation}
Since $\{u=0\} = \{\Gamma(u)=0\}$ and $\Gamma(u)\in\bvo$ has no jump part,~\cite[Proposition~3.92]{AFP} implies that $|D\Gamma(u)|$ vanishes on the set $\{u=0\}$. Taking~\eqref{eq:Pf_NN_1b} into account, we can rewrite~\eqref{eq:Pf_NN_6} as
\begin{equation*}
    e^{-\Gamma(u)} (-\chi_{\{u>0\}}^* \operatorname{div} z + |D\Gamma(u)| - h(u)f) = 0 \text{ as measures in } \Omega.
\end{equation*}
Given that the inequality~\eqref{eq:Pf_NN_1a} holds and $e^{-\Gamma(u)}>0$ $\mathcal{H}^{N-1}$-$\ae$ in $\Omega$, the above identity implies that~\eqref{eq:NN_dist} holds.
\medskip

\textbf{Step 7.} 
The proof of the boundary condition~\eqref{eq:NN_bound} proceeds as in the proof of Theorems~\ref{th:Exist} and~\ref{th:Exist2}. However, since in the current setting we only have $z\in \DM_{\rm loc}(\Omega)$, to apply the Green identity~\eqref{eq:Green} in~\eqref{eq:prov16} it is necessary to verify that $G_\delta(T_{1+\delta}(u))\in L^1(\Omega, \operatorname{div}z)$. This follows immediately by choosing $G_\delta(T_{1+\delta}(u))$ as test function in~\eqref{eq:NN_dist} and noting that $G_\delta(T_{1+\delta}(u)) \chi_{\{u>0\}}^* = G_\delta(T_{1+\delta}(u))$ holds $\mathcal{H}^{N-1}$-$\ae$ in $\Omega$.
\end{proof}

Finally, we stress that the boundedness results established in Section~\ref{sec:Bound} also hold within the framework of Definition~\ref{def:sol_nonneg} ($f\geq 0$ and $h(0)=\infty$) with identical proofs. In contrast, the uniqueness of solution is not expected in this setting. Indeed, when $g\equiv 0$, this non-uniqueness phenomenon has been illustrated through several examples in~\cite[Section~6]{DecGiSe}.

\addtocontents{toc}{\protect\setcounter{tocdepth}{-1}}

\section*{Acknowledgements}

\addtocontents{toc}{\protect\setcounter{tocdepth}{1}}

The author is supported by
the FPU predoctoral fellowship of the Spanish Ministry of Universities (FPU21/04849).

\end{document}